\documentclass{amsart}
\usepackage{amsfonts,amscd,amssymb,amsmath,amsthm}
\usepackage{graphicx,caption,subcaption,mathrsfs,appendix}
\usepackage{xparse,tikz}
\usepackage{epstopdf}
\usepackage[T1]{fontenc}

\usepackage{ctable,color}
\usetikzlibrary{arrows,chains,matrix,positioning,scopes,patterns}

\begin{document}

\newtheorem{theorem}{Theorem}[section]
\newtheorem{lemma}[theorem]{Lemma}
\newtheorem{lm}[theorem]{Lemma}
\newtheorem{corollary}[theorem]{Corollary}
\newtheorem{conjecture}[theorem]{Conjecture}
\newtheorem{cor}[theorem]{Corollary}
\newtheorem{proposition}[theorem]{Proposition}
\newtheorem{prop}[theorem]{Proposition}
\theoremstyle{definition}
\newtheorem{definition}[theorem]{Definition}
\newtheorem{example}[theorem]{Example}
\newtheorem{claim}[theorem]{Claim}
\newtheorem{remark}[theorem]{Remark}


\newcommand{\R}{\mathbb{R}}
\newcommand{\C}{\mathbb{C}}
\newcommand{\Z}{\mathbb{Z}}
\newcommand{\Q}{\mathbb{Q}}
\newcommand{\E}{\mathbb E}
\newcommand{\N}{\mathbb N}
\newcommand{\X}{\mathbf X}

\newcommand{\del}{\ensuremath{\partial}}
\newcommand{\Def}{\ensuremath{:=}}

\newcommand{\TODO}[1]{\textbf{TODO:}{#1}}


\renewcommand{\Pr}{\mathbb{P}}
\newcommand{\as}{\text{a.s.}}
\newcommand{\Prob}{\Pr}
\newcommand{\Exp}{\mathbb{E}}
\newcommand{\expect}{\Exp}
\newcommand{\1}{\mathbf{1}}
\newcommand{\prob}{\Pr}
\newcommand{\pr}{\Pr}
\newcommand{\filt}{\mathscr{F}}
\DeclareDocumentCommand \one { o }
{%
\IfNoValueTF {#1}
{\mathbf{1}  }
{\mathbf{1}\left\{ {#1} \right\} }%
}
\newcommand{\Bernoulli}{\operatorname{Bernoulli}}
\newcommand{\Binomial}{\operatorname{Binom}}
\newcommand{\Beta}{\operatorname{Beta}}
\newcommand{\Binom}{\Binomial}
\newcommand{\Poisson}{\operatorname{Poisson}}
\newcommand{\Exponential}{\operatorname{Exp}}
\newcommand{\Unif}{\operatorname{Unif}}
\newcommand{\lawequals}{\overset{\mathcal{L}}{=} }


\newcommand{\Deg}{\operatorname{deg}}
\DeclareDocumentCommand \vso { o }
{%
\IfNoValueTF {#1}
{\mathcal{V}  }
{\mathcal{V}\left( {#1} \right) }%
}

\DeclareDocumentCommand \deg { O{ } }
{ \operatorname{deg}_{ #1 }}
\newcommand{\oneE}[2]{\mathbf{1}_{#1 \leftrightarrow #2}}
\newcommand{\ebetween}[2]{{#1} \leftrightarrow {#2}}
\newcommand{\noebetween}[2]{{#1} \centernot{\leftrightarrow} {#2}}
\newcommand{\Gap}{\ensuremath{\tilde \lambda_2 \vee |\tilde \lambda_n|}}
\newcommand{\dset}[2]{\ensuremath{ e({#1},{#2})}}


\newcommand{\Htwo}{ \mathbb{H}}
\newcommand{\Etwo}{ \mathbb{E}}
\newcommand{\GWT}{ \mathcal{T}}
\newcommand{\HB}{ B_{\Htwo}}
\newcommand{\MB}{ B_{\mathbb{M}}}
\newcommand{\dH}{ d_{\Htwo}}
\newcommand{\dM}{ d_{\mathbb{M}}}
\newcommand{\diamH}{ \operatorname{diam}_{\Htwo} }
\newcommand{\dE}{d_{\Etwo}}
\DeclareDocumentCommand \Vol { O{G} }{ \operatorname{Vol}_{#1}}
\newcommand{\VolH}{ \operatorname{Vol}_{\Htwo}}
\newcommand{\VolE}{\operatorname{Vol}_{\Etwo}}
\newcommand{\VolV}{\operatorname{Vol}_{\HPV}}
\newcommand{\VolD}{\operatorname{Vol}_{\HPD}}
\newcommand{\VolM}{ \operatorname{Vol}_{\Htwo}}
\newcommand{\convH}{ \operatorname{conv}_{\Htwo}}
\newcommand{\HPV}{ \mathscr{V}^\lambda}
\newcommand{\HPD}{ \mathscr{D}^\lambda}
\newcommand{\GF}{\mathcal{G}_f}
\newcommand{\DeltaH}{\Delta_{\Htwo}}
\newcommand{\DeltaE}{\Delta_{\Etwo}}
\newcommand{\angleH}{\angle_{\Htwo}}
\newcommand{\angleE}{\angle_{\Etwo}}

\newcommand{\CDisc}{ \operatorname{CD}_{\Htwo}}
\newcommand{\CCtr}{ \operatorname{CC}_{\Htwo}}

\newcommand{\PPP}{ \Pi^\lambda }

\DeclareDocumentCommand \TFX { O{i} O{\pi_{\X}} O{f}  }
{ \Delta_{ {#2}, {#3} }({#1})}

\DeclareDocumentCommand \Filt { O{n} }
{ \mathscr{F}_{{#1}} }

\DeclareDocumentCommand \isol { O{i} O{S} }
{ \Delta_{#1} \left( {#2} \right) }

\newcommand{\aec}{i^*}
\DeclareDocumentCommand \islands {  O{i} }
{ \operatorname{IS}_{ {#1 }} }

\newcommand{\BHF}{ \mathcal{B}}

\title[Anchored expansion and Poisson Voronoi]{Anchored expansion, speed, and the hyperbolic Poisson Voronoi tessellation}
\author{Itai Benjamini}
\address{Department of Mathematics, Weizmann Institute of Science}
\email{itai.benjamini@weizmann.ac.il}
\author{Elliot Paquette}
\address{Department of Mathematics, Weizmann Institute of Science}
\email{paquette@weizmann.ac.il}
\author{Joshua Pfeffer}
\address{Harvard University}
\email{pfeffer@college.harvard.edu}
\thanks{
EP gratefully acknowledges the support of NSF Postdoctoral Fellowship DMS-1304057.
This work was largely performed while JP was a participant in Kupcinet-Getz Summer science school at WIS.
}
\date{\today}
\maketitle

\begin{abstract}
  We show that random walk on a stationary random graph with positive anchored expansion and exponential volume growth has positive speed.  We also show that two families of random triangulations of the hyperbolic plane, the hyperbolic Poisson Voronoi tessellation and the hyperbolic Poisson Delaunay triangulation, have $1$-skeletons with positive anchored expansion.  As a consequence, we show that the simple random walks on these graphs have positive speed.  We include a section of open problems and conjectures on the topics of stationary geometric random graphs and the hyperbolic Poisson Voronoi tessellation.
\end{abstract}

\section{Introduction}
\label{sec:anch}
A rooted, locally finite, unlabeled random graph $(G,\rho)$ is called \emph{stationary} if the distribution of $(G,\rho)$ is the same as the distribution of $(G,X_1)$ where $X_1$ is simple random walk on $G$ started from $\rho$ after $1$ step.  Such a graph is called \emph{reversible} if in addition the birooted equivalence class $(G, X_0, X_1)$ has the same distribution as $(G,X_1,X_0).$  Stationary random graphs enjoy many of the same properties of transitive graphs, which they generalize.  For example, in \cite{BenjaminiCurien} it is shown that a stationary random graph of subexponential growth is almost surely Liouville.  In \cite{Lasso}, the converse is shown under the additional assumption of ergodicity.

The property of reversibility is closely tied to the more familiar notion of unimodularity.  Let $P$ be the law of a unimodular random graph $(G,\rho)$ for which $\Exp \deg(\rho) < \infty.$  Let $Q$ be the law on rooted graphs which is absolutely continuous to $P$ and has Radon-Nikodym derivative $\frac{dQ}{dP} = \frac{\deg(\rho)}{\Exp \deg(\rho)}.$  Then it can be checked that $P$ is unimodular if and only if $Q$ is reversible (see \cite[Proposition 2.5]{BenjaminiCurien}).  Hence statements that hold almost surely for $P$ hold almost surely for $Q,$ and provided the degree of $\rho$ is almost surely positive, statements that hold almost surely for $Q$ hold almost surely for $P.$  

For a finite subset of vertices $S \subseteq \vso[G],$ let $\Vol(S)$ be the sum of degrees of the vertices of $S,$ and let $|\del S|$ denote the number of edges that have exactly one terminus in $S.$  A locally finite connected, rooted graph $(G,\rho)$ is said to have \emph{positive anchored expansion} if
\begin{equation}
	\aec(G) \Def \liminf_{\substack{|S| \to \infty \\ \rho \in S \\ G\vert_S \text{connected}}} \frac{|\del S|}{\Vol(S)} > 0.
	\label{eq:ae}
\end{equation}
Note that, in a connected graph, $\aec(G)$ is independent of the root chosen.  
This definition is a natural relaxation of the condition of having a positive edge isoperimetric constant, in which the $\liminf$ is replaced by the infimum over all finite $S.$ 

There are many examples of random graph processes that fail to have positive edge isoperimetric constant but do have positive anchored expansion, such as 
supercritical Galton-Watson trees conditioned on nonextinction~\cite{LyonsPemantlePeres}.
A key feature of positive anchored expansion is that it is stable under random perturbations: for $p$ sufficiently large, $p$-Bernoulli bond or site percolation on a graph with positive anchored expansion gives a graph with positive anchored expansion ~\cite{BenjaminiLyonsSchramm,BLPS,ChenPeres}.

There are perhaps three fundamental known consequences of positive anchored expansion.  The earliest follows from Thomassen~\cite{Thomassen}, who gives a general criterion on the isoperimetric profile of a graph that implies that simple random walk is transient; as a special case, his result shows that positive anchored expansion implies that simple random walk is transient.  The second, due to Vir\'ag~\cite{Virag}, is that under the additional assumption of bounded degree, simple random walk almost surely has positive liminf speed, i.e.,
\begin{equation*}
	\liminf_{k \to \infty} \frac{d(\rho,X_k)}{k} > 0, 
	\label{eq:positivespeed}
      \end{equation*}
where $d$ is the graph metric.  
The third result, also due to Vir\'ag~\cite{Virag}, is a heat kernel bound that states that, for every vertex $x$, there is an $N$ so that for all $n > N$ and all vertices $y$,
\begin{equation}
  \label{eq:hk}
	p^n(x,y) < e^{-\alpha n^{1/3}}.
      \end{equation}
Further consequences of positive anchored expansion, such as the existence of a phase transition in the Ising model with nonzero external field are shown in \cite{HaggstromSchonmannSteif}.

We will show that a stationary graph with positive anchored expansion and exponential growth has positive speed, thus effectively trading the bounded degree assumption of~\cite{Virag} for the assumptions of stationarity and exponential growth.
Let $B(\rho,r)$ denote the closed ball of radius $r$ around $\rho$ in the graph metric.
\begin{theorem}
  \label{thm:speed}
  Let $(G,\rho)$ be a stationary random graph so that:
  \begin{enumerate}
    \item $(G,\rho)$ has positive anchored expansion almost surely and
    \item $\limsup_{r \to \infty} |B(\rho,r)|^{1/r} < \infty$ almost surely.
  \end{enumerate}
  Then simple random walk $X_k$ started from $\rho$ has positive speed, i.e., the limit 
  \[
	s = \lim_{k \to \infty} \frac{d(\rho,X_k)}{k} 
  \]
  exists and $s >0$ almost surely.
\end{theorem}
\noindent Note that further assumptions are needed to say that $s$ is non-random.

For certain classes of graphs, positive speed is enough to ensure the existence of nonconstant bounded harmonic functions (e.g. planar, bounded degree graphs \cite{BenjaminiSchramm96}).  We show that a stationary random graph with nonconstant bounded harmonic functions must have an infinite-dimensional space of such functions.
\begin{theorem}
  A stationary random graph $(G,\rho)$ either has a $1$-dimensional space of bounded harmonic functions or an infinite-dimensional space of bounded harmonic functions.
	\label{thm:harmonics}
\end{theorem}
Work of~\cite{Lasso} also shows that under mild conditions on a stationary graph, positive speed implies the existence of an infinite-dimensional space of bounded harmonic functions.

\subsection*{Hyperbolic Poisson Voronoi tessellation}

Theorems~\ref{thm:speed} and \ref{thm:harmonics} are tailor-made for random graphs that arise as invariant random perturbations of homogeneous spaces, such as invariant percolation on non-amenable Cayley graphs.    We will show how this framework can be applied to a random discretization of Riemannian symmetric space.  Specifically, we will consider the hyperbolic plane $\Htwo.$

Let $\PPP$ denote a Poisson point process on $\Htwo$ with intensity given by a multiple $\lambda > 0$ of its invariant volume measure, which is conditioned to have a point at $0$.  The exact normalization for the area measure is given at the start of Section~\ref{sec:hpv}.  

The hyperbolic Poisson Voronoi tessellation of $\Htwo$ is a polygonal cell complex, with each cell containing exactly one point in $\PPP$; this point is referred to as the \emph{nucleus} of the cell. For a point $p_0 \in \PPP$, the cell with nucleus $p_0$ is given by 
\[
\left\{ z \in \Htwo : \dH(z,p_0) = \min_{p \in \PPP} \dH(z,p)\right\},
\]
where $\dH$ denotes distance in the hyperbolic metric.  

From the point process $\PPP$, we can also construct the Poisson Delaunay complex, a polygonal cell complex with vertices $\PPP$, and a hyperbolic geodesic between two vertices if and only if the corresponding Voronoi cells share a boundary edge.
See Figure~\ref{fig:simulation} for a simulation of these tessellations.

It is elementary to see that, in the Poisson Delaunay complex, all the faces are almost surely triangles; hence, we will refer to it as the Delaunay triangulation.  Further, it can be checked that a triple of points $\{x,y,z\}$ in $\PPP$ forms a triangle in the Delaunay triangulation if and only if all three points lie on the boundary of a finite hyperbolic disk in $\Htwo$ whose interior contains no points of $\PPP$.

\begin{figure}
  \begin{subfigure}[t]{0.45\textwidth}
    \includegraphics[width=\textwidth]{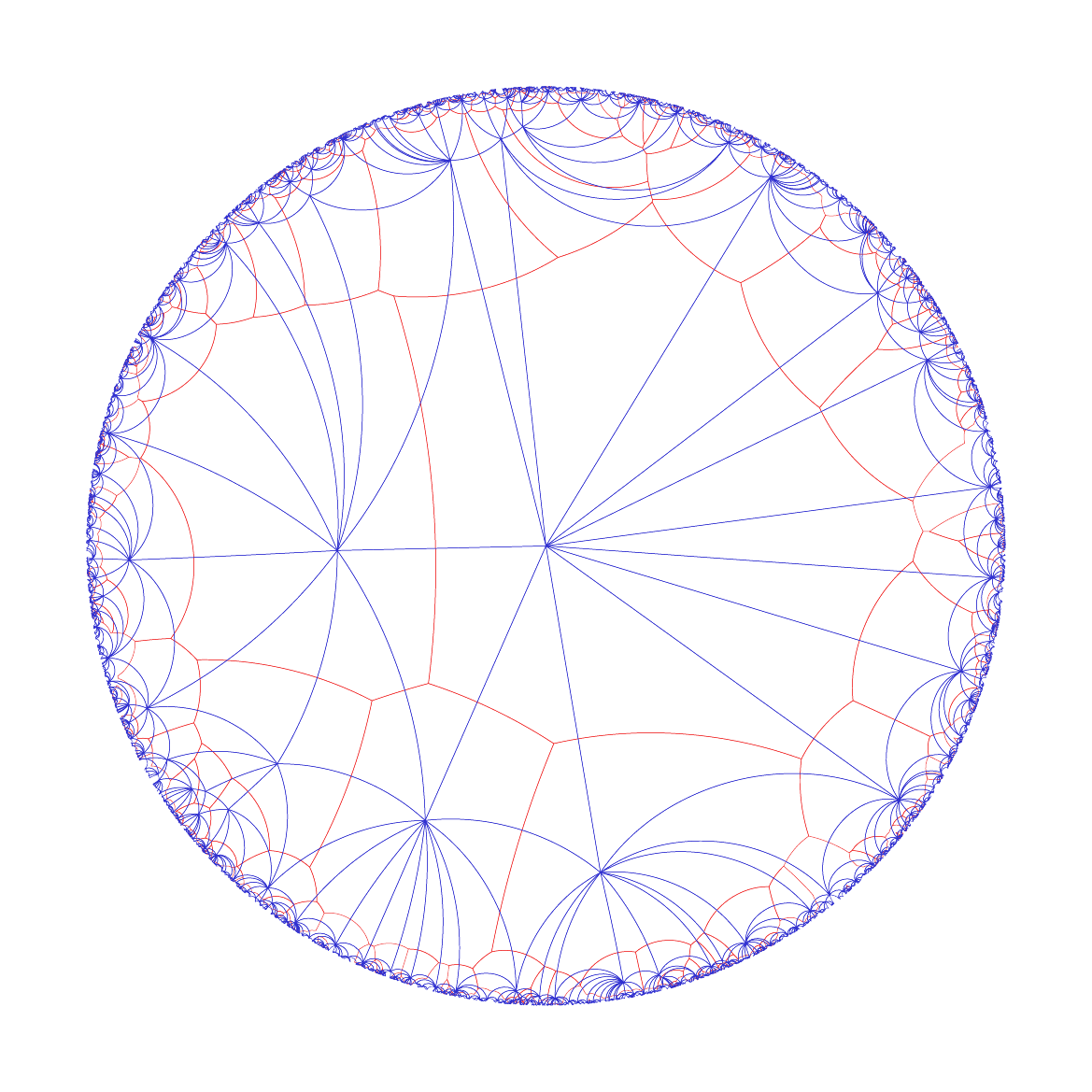}
    \caption{$\lambda=0.2$ and $r = 0.9995.$}
  \end{subfigure}
  \begin{subfigure}[t]{0.45\textwidth}
    \includegraphics[width=\textwidth]{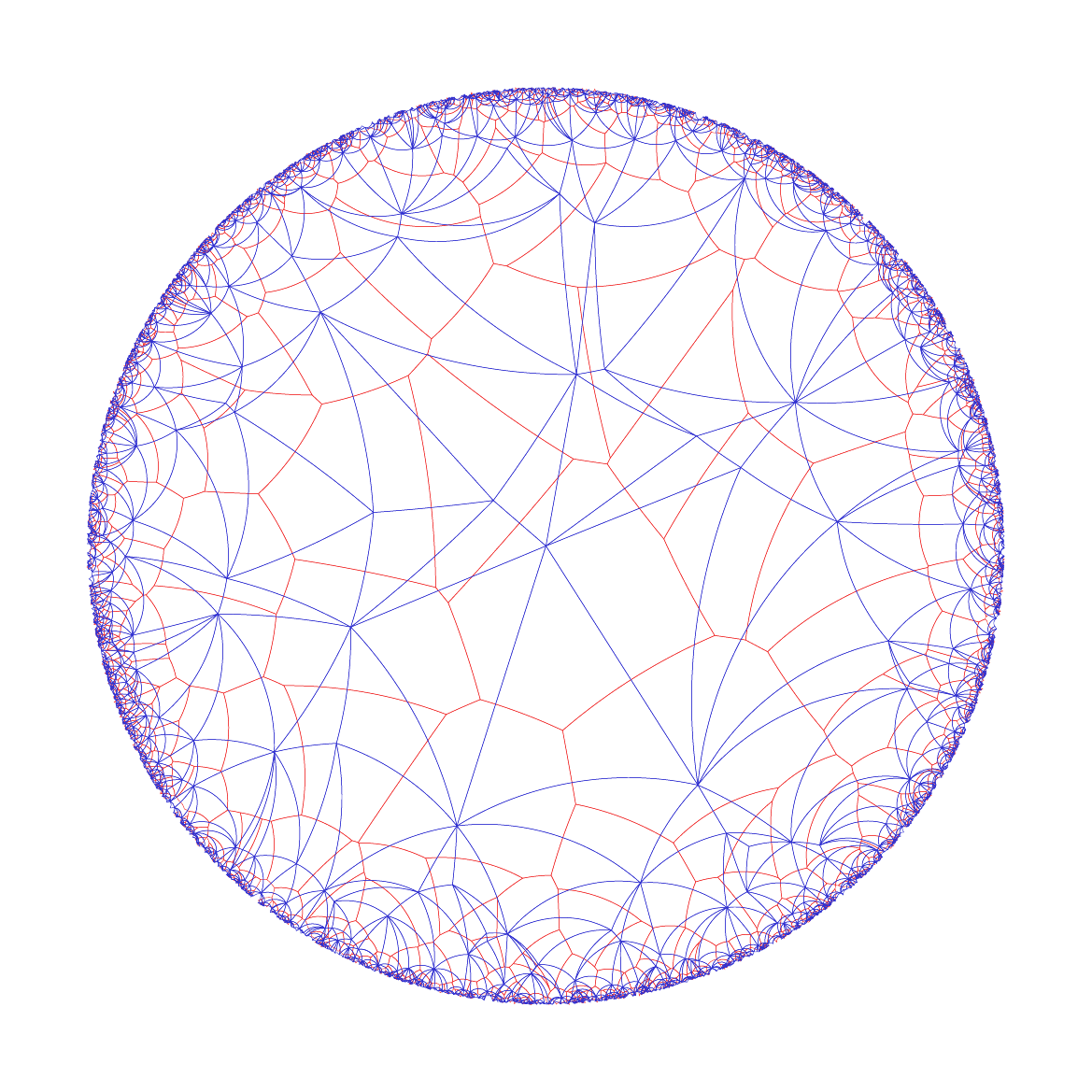}
    \caption{$\lambda = 1$ and $r=0.9975.$} 
  \end{subfigure}
  \caption{A rendering of the Poisson Voronoi tessellation (red) for various $\lambda$ with superimposed Delaunay triangulation (blue).  In each simulation, the process $\PPP_r$ on $\HB(0,r)$ is sampled with intensity $\lambda$ times hyperbolic area measure.  In electronic versions of this figure, it is possible to get more detail by zooming in on the picture.  Note that for smaller $\lambda$ the average degree increases.  
  }
  \label{fig:simulation}
\end{figure}

The probability that a disk in $\Htwo$ contains no points of $\PPP$ is exponentially small in the volume of the disk.  As the volume of a disk centered at a fixed point in $\Htwo$ grows exponentially in the radius, the maximal diameter of a Delaunay triangle incident to that point has a very sharp tail:
\begin{lemma}
	Let $S_0$ be the union of the Delaunay triangles with one vertex equal to $0.$  There is a constant $C>0$ so that for all $r > 0$
	\[
		\Pr\left[
			S_0 \not\subset \HB(0,r)	
		\right] \leq Ce^{3r/4 - \lambda \pi e^{r/4}},
	\]
	where $\HB(x,r)$ is the hyperbolic ball centered at $x$ of radius $r.$
	\label{lem:triangletail}
\end{lemma}
Note that, by the transitivity of the automorphism group of $\Htwo$ and the invariance of the intensity of $\PPP,$ this lemma also holds if $0$ is replaced by any other fixed point in $\Htwo$ (see Remark~\ref{rem:same}). 

For both the Poisson Voronoi tessellation and the Poisson Delaunay triangulation, we consider the dual graph, in which each cell becomes a vertex of the dual graph, and two vertices of the dual graph are adjacent if and only if the corresponding cells of the tessellation share a boundary arc.  Let $\HPV$ and $\HPD$ denote the dual graphs of the hyperbolic Poisson Voronoi tessellation and the hyperbolic Poisson Delaunay triangulation, respectively.  Note that $\HPV$ is geometrically realized as the $1$-skeleton of the Delaunay complex, and $\HPD$ as the $1$-skeleton of the Voronoi tessellation. 
 Except where absolutely necessary, we will not distinguish between the abstract graphs $\HPV$ and $\HPD$ and their geometric realizations as $1$-skeletons.

Our main theorem about these tessellations is that both have positive anchored expansion almost surely.
\begin{theorem}
	For $G=\HPV$ or $\HPD$ there is a constant $c=c(\lambda) > 0$ so that 
	\(
		\aec(G) > c
	\)
	almost surely.
	\label{thm:hpv}
\end{theorem}

Both graphs can be seen to satisfy the assumptions of Theorem~\ref{thm:speed} without much additional effort.  The fact that both are stationary follows from the extent of the automorphism group. 
 More generally, the following is true:
\begin{proposition}
  Let $\mathbb{M}$ be a Riemannian symmetric space, and fix $x \in \mathbb{M}$.  Let $\mathcal{X}$ be a Poisson point process conditioned to include $x$ and  whose intensity is invariant under symmetries of $\mathbb{M}$.  Denote by $\mathscr{G}$ be the dual graph of the Voronoi tessellation, and let $\rho$ be the vertex of $\mathscr{G}$ embedded at $x$.  Let $P$ be the distribution of $(\mathscr{G},\rho).$  Then
  \begin{enumerate}
    \item $\Exp \deg \rho < \infty$ and
    \item with $\frac{dQ}{dP} = \frac{\deg \rho}{\Exp \deg \rho},$ $(\mathscr{G},\rho)$ is reversible under $Q.$ 
  \end{enumerate}
  \label{prop:stationarity}
\end{proposition}
\noindent Equivalently, $(\mathscr{G},\rho)$ under $P$ is unimodular.  We give the proof in Section~\ref{sec:other}.

\begin{remark}
  \label{rem:same}
  The same is also true for $\HPD$ if we condition $x$ to be a vertex of the $1$-skeleton of the Voronoi tessellation.  In addition, Theorem~\ref{thm:hpv} is easily seen to remain true under this conditioning or indeed under no conditioning.  Let $\Pi_0^\lambda$ be the unconditioned version of $\PPP,$ and let $\Pi_1^\lambda$ be distributed as $\Pi_0^\lambda$ conditioned to have a vertex of the $1$-skeleton of the Voronoi tessellation at $0.$  This is equivalent to requiring the smallest closed disk centered at $0$ that contains a point of $\Pi_0^\lambda$ to have three i.i.d points chosen uniformly from its boundary.  Hence, the processes $(\Pi_0^\lambda, \PPP, \Pi_1^\lambda)$ can be coupled in the following way.  Let $\PPP$ be given by $\Pi_0^\lambda \cup \{ 0 \}$ and let $\Pi_1^\lambda$ be given by $\Pi_0^\lambda \cup \{X,Y\}$ where $X,Y$ two additional random points uniformly distributed from the circle that contains the closest point of $\Pi_0^\lambda$ to $0.$  Since adding or deleting finitely many points does not change $\aec(G)$, Theorem~\ref{thm:hpv} is true for all three point processes.
\end{remark}

Observe that $\HPD$ trivially satisfies the exponential growth condition of Theorem~\ref{thm:speed} since it is $3$-regular.  Hence by Theorem~\ref{thm:speed} (or indeed the original theorem of~\cite{Virag}), simple random walk on $\HPD$ has positive speed.  Since $\HPD$ is bounded degree and planar, it is also non-Liouville \cite{BenjaminiSchramm96}; so, by Theorem~\ref{thm:harmonics}, it has an infinite-dimensional space of bounded harmonic functions.  Finally, in Proposition~\ref{prop:S1convergence}, we show that the random walk converges to a point on the ideal boundary of $\Htwo$ (effectively that as a sequence of points in the Poincar\'e disk, it converges to a point on the unit circle $S^1$).

For $\HPV,$ the exponential growth condition of Theorem~\ref{thm:speed} is given by Proposition~\ref{prop:expgrowth}.  Hence it too has positive speed.  The existence of an infinite-dimensional space of bounded harmonic functions follows from the entropy considerations of~\cite{Lasso} (the result of~\cite{Lasso} could also be applied to $\HPD$).  Like with $\HPD,$ random walk almost surely converges to a point on $S^1.$


\subsection*{Discussion}

We give a criterion for positive speed of simple random walk in terms of positive anchored expansion, which implies transience.  For many classes of stationary random graphs, there is an alternative between recurrence and positive speed.  For example, unimodular random trees of at most exponential growth and bounded degree planar graphs both have that random walk is almost surely recurrent or random walk has positive speed (this follows from combining results of \cite{BenjaminiSchramm96} and \cite{BenjaminiCurien}).  We would expect that bounded degree unimodular Gromov hyperbolic graphs are another such class.  Note that $\HPV$ in a sense barely fails to satisfy any of these conditions.

The hyperbolic Poisson Voronoi tessellation is first considered by~\cite{BenjaminiSchramm01}, in which Bernoulli percolation on $\HPV$ is studied.  There it is shown that Bernoulli percolation undergoes two phase transitions at values $p_c(\lambda) < p_u(\lambda)$ from having only finite clusters, to having infinitely many infinite clusters, to having a unique infinite cluster.  Further, an upper bound is given on $p_c(\lambda)$ that suggests (and led the first author and Oded Schramm to conjecture) that $p_c(\lambda) \to \tfrac 12$ as $\lambda \to \infty.$

Indeed, it is the case that as $\lambda \to \infty,$ all finite neighborhoods of $0$ in $\HPV$ will converge in law to those of the Euclidean Poisson Voronoi tessellation.  This is because increasing $\lambda$ is tantamount to decreasing the curvature.  This curvature effect can also be seen in how the average degree (appropriately defined) decreases as $\lambda$ increases (see Figure~\ref{fig:simulation} for an illustration).   

The hyperbolic Poisson Voronoi tessellation is one of many planar stochastic triangulations with hyperbolic characteristics.  The planar stochastic hyperbolic triangulation of~\cite{Curien} is formed by a peeling process that ensures a type of spatial Markov property. The Markovian hyperbolic triangulation~\cite{CurienWerner} gives a triangulation of the hyperbolic plane by infinite triangles. There are also half-plane versions~\cite{AngelRay,AngelNachmiasRay} with another type of Markov property which share many of the same features.

Though the hyperbolic Poisson Voronoi tessellation does not have a domain Markov property, the diameter of any given Voronoi cell has a subexponential tail (cf. Lemma~\ref{lem:triangletail}).  Hence, there is a very strong sense in which the graph is local, i.e., the law of a small neighborhood of a graph depends only on the portion of $\PPP$ in a small ball around that point.  This leads to rapid decorrelation between a neighborhood of any vertex of $\HPV$ and a neighborhood of a vertex that is far away.  

Consequently, it is straightforward to show that the hyperbolic Poisson Voronoi tessellation arises as a local limit of finite triangulations.
This was observed earlier by the first author and Oded Schramm before the notion of local limit was codified (see~\cite[``Hyperbolic Surfaces'' proof of Theorem 6.2]{BenjaminiSchramm01}); we also give a proof that $\HPV$ and $\HPD$ are local limits in Proposition~\ref{prop:sofic}. 

For approaches to finite random triangulations of the full plane that use enumerative techniques, this property is comparably quite difficult; indeed this question is open for the plane and half-plane stochastic hyperbolic triangulations.  However, the large scale behavior of all of these triangulation models should be similar, and it would be interesting to know if there is a specific sense in which the hyperbolic Poisson Voronoi tessellation is absolutely continuous with respect to the planar stochastic hyperbolic triangulation in the hyperbolic regime.

\subsection*{Open questions}

We believe that this analysis has only scratched the surface of what can be asked about the hyperbolic Poisson Voronoi tessellation.  First, the hyperbolic Poisson Voronoi tessellation generalizes immediately to higher-dimensional hyperbolic spaces, for which it should still be the case that various dual graphs have positive anchored expansion.  We believe the same approach used here could be adapted to that case with the principal missing component being the analogue of Proposition~\ref{prop:geometry}.
\begin{conjecture}
  Let $\mathbb{H}^d$ denote $d$-dimensional hyperbolic space, and let $\PPP$ be a Poisson process with invariant intensity measure.  The dual graph of the Voronoi tessellation of $\PPP$ has positive anchored expansion.
\end{conjecture}

More could be said about the speed, $s(\lambda) = \lim_{k \to \infty} d_G(X_k,\rho)/k.$  Note that our result does not show that $s(\lambda)$ is deterministic, which we expect it should be.  As increasing $\lambda$ has the effect of decreasing curvature, we expect that $s(\lambda)$ should be continuous and strictly monotone decreasing with $s(0^+) = 1$ and $s(\infty^-) = 0.$  

By Theorem \ref{thm:harmonics} and \cite{Lasso}, we know that the space of bounded harmonic functions of the random walk on $\HPV$ is infinite-dimensional, but considerably more could be said.  For example, it would be nice to know that $\HPV$ has bounded harmonic functions with finite Dirichlet energy (c.f. \cite[Theorem 1.1]{BenjaminiSchramm96}) and that the wired uniform spanning forest is almost surely $1$-ended (c.f. \cite[Theorem 7.2]{AldousLyons}).

Recall that the space of bounded harmonic functions can be endowed with a measure to make it isomorphic to the Poisson boundary (see~\cite{Kaimanovich02} for relevant background).  
For simple random walk on $\HPV$ or $\HPD,$ we show that simple random walk considered as a sequence in the Poincar\'e disk converges almost surely to a point of $S^1$ in the topology of $\C$ (see Proposition~\ref{prop:S1convergence}).
There are many general results about boundary convergence of random walk that are nearly applicable to $\HPV$ and $\HPD.$
For example, \cite[Theorem 1.1]{BenjaminiSchramm96} or the recent work of \cite{Georgakopoulos, ABGGN} nearly apply to $\HPV.$   
This leads us to believe that these results could be extended to cover graphs like $\HPV$. For example in the latter work, is it possible to replace the bounded degree assumption by a stationary assumption as we have done here?
\begin{conjecture}
  Let $(G,\rho)$ be a transient stationary random graph that is almost surely planar.  
  Then $(G,\rho)$ embeds as a stationary subgraph of the Poincar\'e disk and 
  for almost every realization of $(G,\rho)$
  simple random walk almost surely converges to a point on $S^1$ in the topology of $\C.$
\end{conjecture}

Using the convergence of simple random walk to a point on $S^1$, the Poisson boundary can be naturally identified with the unit circle $S^1,$ together with a measure $\nu_0$ which can be considered as the harmonic measure at $\infty$.  
Precisely, for Borel $A \subset S^1,$ $\nu_0(A)$ is the probability that random walk converges in the Euclidean metric to a point in $A$. 
For a reference point, the Poisson boundary of hyperbolic Brownian motion on the Poincar\'e disk started at $0$ is naturally identified with Lebesgue measure on $S^1$.  
Given the dimension drop phenomenon observed for harmonic measure on infinite supercritical Galton-Watson trees~\cite{LyonsPemantlePeres}, we expect this is not the case here:
\begin{conjecture}
For almost every realization of $\HPV,$ $\nu_0$ is singular with respect to Lebesgue measure on $S^1.$
\end{conjecture}

While $\HPV$ and $\HPD$ are hyperbolic on a large scale, they can not satisfy many types of uniform hyperbolic properties, such as Gromov hyperbolicity, by virtue of every finite planar graph embedding into $\HPV.$ That said, there is still room to characterize which qualitative features of hyperbolic graphs are present.  For example, it is natural to ask how nearly geodesics in $\HPV$ match geodesics from the hyperbolic plane.  
\begin{conjecture}
Consider connecting the Voronoi cells containing $-r$ and $r$ by a geodesic in $\HPV,$ and let $D_r$ be the graph distance from this geodesic to the cell containing $0;$ then $D_r$ a tight family of variables.
\end{conjecture}
In a similar spirit, is it the case that for every vertex $x$ from $\HPV$ there is a $\delta_x$ almost surely finite so that all triangles with one vertex at $x$ are $\delta_x$-thin?  This is a natural ``anchored'' analogue of Gromov hyperbolicity.

While Theorem~\ref{thm:speed} shows positive speed for a large class of stationary random graphs, it is natural to ask whether the volume growth condition is necessary.  We wonder if it can be weakened or removed entirely.  As for the conclusions, it is natural to ask whether or not the heat kernel bound~\eqref{eq:hk} that holds for bounded degree graphs with positive anchored expansion extends to stationary random graphs with positive anchored expansion.

\section{Speed}


We begin with the proof of the criteria for positive speed, Theorem~\ref{thm:speed}.  Following Vir\'ag~\cite{Virag}, for any $S \subseteq \vso[G]$ finite, let $\isol[i][S] \Def i|S| - |\del S|.$  Define an \emph{isolated $i$-core} to be a finite set of vertices $S$ so that for all $A \subsetneq S$ $\isol[i][S] > \isol[i][A].$  As it is possible to take $A = \emptyset$ in this definition, it follows that an isolated $i$-core must have $\isol[i][S] > 0.$ 

Let $\islands$ denote the union of all isolated $i$-cores in $G.$  The following are easily checked:
\begin{proposition}
	For any locally finite, connected, simple graph $G,$
	\begin{enumerate}
		\item Any finite union of isolated $i$-cores is again an isolated $i$-core.
		\item If $\aec(G) > i > 0,$ then any vertex is contained in at most finitely many connected isolated $i$-cores.
		\item If $i = \aec(G) > 0,$ then $G \setminus \islands$ is a (nonempty) graph with edge isoperimetric constant $i.$
	\end{enumerate}
	\label{prop:basics}
\end{proposition}
\noindent See Corollary 3.2 and Proposition 3.3 of~\cite{Virag} for a proof.

Positive speed is essentially due to the fact that the induced walk on $G \setminus \islands$ has positive speed.  All that needs to be checked is that the walk returns to $G \setminus \islands$ frequently enough, and here we use stationarity in an essential way.  Indeed, without stationarity and without bounded degree, it can be shown that positive anchored expansion is insufficient for positive speed.

Let $T_1,T_2,\ldots$ be the times at which $X_k \in G \setminus \islands.$  We would like to show that $T_n / n \to c < \infty.$
Define $h_i\left( G, \rho \right)$ to be the indicator for the event that $\rho$ is in $G \setminus \islands.$  By the ergodic theorem, we have that
\begin{equation}
	\lim_{n \to \infty} \frac{1}{n+1} \sum_{j=0}^n h_i ( G, X_j) 
	=
	\Pr\left[ \rho \in G \setminus \islands \mid \filt \right],
	\label{eq:occupy}
\end{equation}
where $\filt$ is the invariant $\sigma$-algebra.
We first show that due to stationarity, this probability can not be $0$ for $i$ sufficiently small.  
\begin{lemma}
	\label{lem:occupy}
	Let $(G,\rho)$ be stationary.
	For any fixed $i > 0,$
let $\mathcal{E}_i \in \filt$ be the event that
\begin{align*}
  \Pr\left[ \rho \in G \setminus \islands \mid \filt \right] &= 0 \text{ and }\\
\Pr\left[ \aec(G) > i \mid \filt \right] &> 0.
\end{align*}
Then $\Pr\left( \mathcal{E}_i \right) = 0.$
\end{lemma}
\begin{proof}
By Proposition~\ref{prop:basics} when $\aec(G) > i$ there are only finitely many connected isolated $i$-cores containing $\rho.$ 
On the other hand, by stationarity for any $\omega \in \mathcal{E},$ $\Pr\left[ X_k \in G \setminus \islands \mid \filt \right]=0$ for all fixed $k \geq 0.$  In particular, for all $k,$ $X_k$ is contained in some connected, isolated $i$-core, $C_k.$  Let $V_k = \left\{ X_0,X_1,X_2,\ldots,X_k \right\}.$  As the graph has positive anchored expansion almost surely, random walk is almost surely transient and so $|V_k| \to \infty$ almost surely.  As isolated $i$-cores are all finite, the sequence $C_0, C_0 \cup C_1, C_0 \cup C_1 \cup C_2, \ldots$ almost surely contains infinitely many connected isolated $i$-cores containing $\rho,$ so that $\Pr( \mathcal{E}_i ) = 0.$ 
\end{proof}

\noindent The proof of Theorem~\ref{thm:speed} is now a simple consequence.
\begin{proof}[Proof of Theorem~\ref{thm:speed}]
  By assumption, $\aec(G) > 0$ almost surely.  Hence by Lemma~\ref{lem:occupy} and by letting $i$ run over $\{1/k\}_{k=1}^\infty,$ we may restrict to considering realizations of $\left( G, \left( X_i \right)_0^\infty \right)$ and $i>0$ for which
  \(
  \Pr\left[ \rho \in G \setminus \islands \mid \filt \right] > 0.
  \)
Recalling \eqref{eq:occupy}, we know that
\[
	\lim_{n \to \infty} \frac{1}{n+1} \sum_{j=0}^n h_i ( G, X_j) 
	=
	\Pr\left[ \rho \in G \setminus \islands \mid \filt \right] > 0.
\]
Hence, it also follows that
\[
	\lim_{n \to \infty} \frac{T_n}{n} = \frac{1}{\Pr\left[ \rho \in G \setminus \islands \mid \filt \right]} < \infty.
\]

Consider the induced random walk $Y_k =  X_{t_k}$ on $G \setminus \islands.$  From the Cheeger inequality on $G \setminus \islands$ (Proposition~\ref{prop:basics}), the induced random walk satisfies a heat kernel estimate of the form 
\[
  \sup_{ y \in G \setminus \islands}
  \Pr \left[
	Y_k = y \,|\, \sigma(G,\rho)
  \right] \leq q^k,
\]
for some $q = q(G,\rho) < 1.$  From the almost sure subexponential growth assumption, we have that 
\(
 M = \sup_{r \in \mathbb{N}} |B(\rho,r)|^{1/r}
\)
is almost surely finite.  Applying a union bound, we get that
\[
  \Pr \left[
	d_G( Y_k, \rho) \leq ck \,|\, \sigma(G,\rho)
      \right] \leq M^{ck} q^k.
\]
Thus, adjusting $c > 0$ to be sufficiently small, we get by Borel-Cantelli that $d_G(Y_k,\rho) > ck$ infinitely often almost surely, and hence we have
\[
  \liminf_{k \to \infty} \frac{d_G(\rho,X_{T_k})}{k} > c' > 0
\]
for some other $c'.$ 

There only remains to show that this implies the desired result on the speed of $X_k.$  Note that this result and Lemma~\ref{lem:occupy} together give that 
\[
  \liminf_{k \to \infty} \frac{d_G(\rho,X_{T_k})}{T_k} > 0.
\]
As the map $j \mapsto d_G(\rho, X_j)$ is $1$-Lipschitz we have that for $T_k \leq j \leq T_{k+1},$
\[
  \frac{d_G(\rho, X_{T_{k+1}})}{T_{k}} - \frac{T_{k+1} - T_k}{T_{k+1}}
  \leq \frac{d_G(\rho, X_j)}{j}.
\]
As $T_k / k$ converges to a positive constant almost surely, we have that $\frac{T_{k+1} - T_k}{T_{k+1}} \to 0$ almost surely, from which it follows that $\liminf_{k\to\infty} d_G(\rho,X_k)/k > 0$ almost surely. 

Meanwhile from the subadditive ergodic theorem~\cite[Theorem 5.3]{Krengel}, the limit $\lim_{k\to \infty} d_G(\rho, X_k)/k$ exists almost surely and is equal to some $\filt$-measurable random variable, which completes the proof.
\end{proof}

We also show that for a stationary graph with bounded harmonic functions, there must be infinitely many.
\begin{proof}[Proof of Theorem~\ref{thm:harmonics}]
  For a connected, locally finite graph $G$, let $\BHF(G)$ denote the space of bounded harmonic functions.  Let $S \subset \BHF(G)$ be the unit ball under the $L^\infty(G)$ norm.
  We will define a mapping from $\varphi: \vso(G) \to S$ that produces a harmonic function almost surely.  Fix $v \in \vso(G)$ and define
  \[
    R = \inf \left\{r \in \mathbb{Z}_{\geq 0} : \min_{ x \in B_G(v,r)} \inf_{ h \in S} |h(x)| < 1 \right\}.
  \]
  If $R = -\infty,$ then $S$ consists of harmonic functions that only take the values $\{\pm 1\}.$  As $G$ is connected, such harmonic functions must be constant.  Hence, if $R=-\infty,$ define $\varphi(v) = 1 \in \BHF(G).$  Otherwise, let $\varphi(v)$ be any arbitrary $h \in S$ with $\min_{ x \in B_G(v,R)} |h(x)| < 1.$  
  
  Note we may pick $\varphi$ to be equivariant with respect to rooted isomorphisms, i.e., if there is an isomorphism of rooted graphs $\psi : (G,\rho_1) \to (G,\rho_2)$ then $\varphi(\rho_1) = \varphi(\rho_2) \circ \psi^{-1}.$  This in particular assures that $\varphi$ is defined on rooted equivalence classes of graphs.  As a consequence, the sequence $(h_k)_{k=0}^\infty = (\varphi(X_k))_{k=0}^\infty$ is stationary.

  For every fixed $r>0,$ we define the variables
  \[
    Y_k^r = \max_{i=1,2,\ldots,r} | h_k(X_{k+i}) - h_k(X_k)|\,;
  \]
which makes $(Y_k^r)_{k=0}^\infty$ stationary. 
  
Suppose that $\BHF(G)$ were finite-dimensional.  Then, there would be a basis 
\[
f_1, f_2, \ldots, f_d
\]
of bounded harmonic functions, which we may take to each have $L^\infty(G)$ norm $1$.  There would also be a collection of $d$ vertices $x_1,x_2,\ldots, x_d$ so that the matrix $F_{i,j} = f_i(x_j)$ is invertible.
Let $h$ be any element of $\BHF(G)$ with $L^\infty(G)$ norm at most $1,$ and expand $h = c_1 f_1 + c_2f_2 + \cdots + c_d f_d.$  As $(c_i)_{i=1}^d$ can be computed by $F^{-1} \bar h,$ where $\bar h = (h(x_j))_{j=1}^d,$ we get that there is an $M > 0$ so that all $h \in \BHF(G)$ with $L^\infty(G)$ norm at most $1$ have coefficients in the $f_j$ basis bounded in absolute value by $M.$ 
  
From the definition of $M,$ we get that
\[
    \lim_{k \to \infty} 
    \sup_{\substack{h \in \BHF(G) \\ \|h\|_\infty \leq 1}}
    \max_{i=1,2,\ldots,r} | h(X_{k+i}) - h(X_k)| 
    \leq
    \lim_{k \to \infty} \max_{\substack{i=1,2,\ldots,r \\
    j=1,2,\ldots, d
    }} M|f_j(X_{k+i}) - f_j(X_k)|.
\]
  As each process $( f_j(X_k) )_{k=0}^\infty$ is a bounded martinagle, it converges almost surely.  Hence, for every fixed $r>0,$ it also follows that
  \[
    \lim_{k \to \infty} \max_{\substack{i=1,2,\ldots,r \\
    j=1,2,\ldots, d
    }} | f_j(X_{k+i}) - f_j(X_k)| = 0.
  \]
almost surely. Thus, we conclude that for every fixed $r > 0,$
\[
  \lim_{k \to \infty} Y_k^r = 0.
\]

Then for every fixed $r >0$ and every $\epsilon > 0,$ we have that 
\[
  \lim_{k \to \infty} \frac{1}{k+1}\sum_{i=0}^k \one[ Y_k^r > \epsilon ]  = 0.
\]
By the ergodic theorem, however, we also have that
\[
  \lim_{k \to \infty} \frac{1}{k+1}\sum_{i=0}^k \one[ Y_k^r > \epsilon ]  = 
  \Pr \left[
    \exists i \leq r ~:~ |h_0(X_i) - h_0(\rho)| > \epsilon
    \mid \filt 
  \right].
\]
Hence we have that $h_0(X_i) = h_0(\rho)$ almost surely, whence $R = -\infty,$ which implies that all bounded harmonic functions are constant.
\end{proof}

\section{Hyperbolic Poisson Voronoi tessellation}
\label{sec:hpv}

Let $\Htwo$ denote the hyperbolic plane.  For a good introduction to hyperbolic geometry and different models of the plane, see~\cite{Kenyon}.  
The Poincar\'e disk model of $\Htwo$ is given by the unit disk $\{ (x_1,x_2) : x_1^2 + x_2^2 < 1 \}$ in $\R^2$ together with the Riemannian metric
\[
  ds^2_{\Htwo} = 4 \frac{dx_1^2 + dx_2^2}{(1 -  x_1^2 - x_2^2)^2}.
\]
From this it follows that hyperbolic area measure is absolutely continuous to Lebesgue measure and has density given by
\[
  dA_{\Htwo} = 4 \frac{dx_1 dx_2}{(1 -  x_1^2 - x_2^2)^2}.
\]
We let $\PPP$ be the Poisson process on the open unit disk with intensity $\lambda \cdot dA_{\Htwo}.$ 
We always condition $\PPP$ to contain $0$.  Note that by Lemma~\ref{lem:triangletail}, adding this vertex alters at most finitely many edges of $\HPV$ or $\HPD$ almost surely, and hence positive anchored expansion holds for the process without the additional point.

Our proof of Theorem~\ref{thm:hpv} relies on one key observation about the Delaunay triangulation: the hyperbolic area of a contiguous collection of triangles that is adjacent to the origin is on the order of the number of triangles considered.
Specifically, we call a collection of Delaunay triangles \emph{strongly connected} if they form a connected set in $\HPD;$ equivalently, a collection of Delaunay triangles is connected if and only if the interior of the union of these triangles is connected.  Call them strongly connected to the origin if one of these is a triangle containing the point $0.$  Then, we have the following proposition.

\begin{proposition}
  There is a constant $c > 0$ and a $k_0 > 0$ random so that for all collections of Delaunay triangles $t_1,t_2, \ldots, t_k$ with $k > k_0$ that are strongly connected to the origin and whose union $\cup_{i=1}^k t_i$ is simply connected,
	\[
		\sum_{i=1}^k \VolH(t_i) > ck,
	\]
where for any Borel $U \subset \Htwo,$ $\VolH(U)$ denotes hyperbolic area of $U$.
\label{prop:strongarea}
\end{proposition}

The proof of Theorem~\ref{thm:hpv} is a relatively straightforward consequence of this proposition and following observation:
\begin{lemma}
  For any finite set $S \subset \Htwo,$
  \[
    \VolH(\convH(S)) \leq 4\pi |S|,
  \]
  where $\convH(S)$ denotes the hyperbolic convex hull $S$.
  \label{lem:BE}
\end{lemma}
This additionally holds in higher dimensional hyperbolic spaces, with appropriately chosen constants in place of $4\pi;$ see ~\cite[Theorem 1]{BenjaminiEldan} for a proof.  This convex hull observation nearly immediately implies the expansion result for the Voronoi tessellation: for example, the following is true.
\begin{proposition}
  Suppose $S \subset \Htwo$ is a collection of points that are $1$-separated, and suppose that $G$ is the dual graph Voronoi tessellation with nuclei given by $S.$  Provided that all Voronoi cells are finite and that $G$ is locally finite,  
  \[
    \inf_{ \substack{V \subset S  \\  |V| < \infty} } \frac{ |\partial V|}{ | V | } > 0.
  \]
  \label{prop:prototype}
\end{proposition}
\begin{remark}
This proposition generalizes immediately to higher dimensional hyperbolic spaces.  
\end{remark}
\begin{proof}

  For any $p \in S$ be any point, and let $X_p$ be the Voronoi cell with nucleus $p.$  Let $V \subset S$ be any finite set, and
  let $W$ be the vertices of the $1$-skeleton of the Voronoi tessellation that are contained in $\partial ( \cup_{p \in V} X_p ).$  Note that the induced subgraph of this $1$-skeleton with vertices given by $W$ has minimum degree $2$ and hence $| \partial V | \geq | W|.$ 
  Then $\convH( W ) = \convH( \cup_{p \in V} X_p ).$  Each Voronoi cell $X_p$ must contain $\HB(p,1/2)$ and hence by Lemma~\ref{lem:BE}
  \begin{align*}
4\pi | \partial V | 
&\geq
4\pi | W | \\
&\geq \VolH( \convH(W) ) \\
&= \VolH(\convH( \cup_{p \in V} X_p )) \\
& \geq |V|\VolH(\HB(0,1/2)).
  \end{align*}
\end{proof}

In considering $\HPV$ or $\HPD,$ the philosophy of our approach is similar, and the proof of Theorem~\ref{thm:hpv} from Proposition~\ref{prop:strongarea} is not much more complicated.




\begin{proof}[Proof of  Theorem~\ref{thm:hpv}  from Proposition~\ref{prop:strongarea}]

\noindent \emph{Proof that $\HPV$ has positive anchored expansion.}
Let $k_0$ and $c$ be as in Proposition~\ref{prop:strongarea}.
Suppose that for some $k > k_0$ the vertices $v_1,\ldots,v_k \in \HPV$ form a connected subgraph $S$ of $\HPV,$ one of which is the Voronoi cell with nucleus at $0.$ 
Consider the subcomplex $S'$ of the Delaunay triangulation given by the union of all triangles containing vertices of $S$.  
Then every vertex of $S$ is in the interior of the union of triangles formed by $S'.$  Moreover, the vertices of $S$ are exactly the interior vertices of the complex $S'.$  The induced subgraph of $S'$ on the vertices of $S$ is connected, as these vertices were connected in $\HPV,$ and hence $S'$ has connected interior.  Equivalently, $S'$ is a strongly connected collection of triangles.

Observe that the convex hull of the vertices of $S'$ is the same as the convex hull of the vertices of $S' \backslash S$, as every vertex of $S$ is in the interior of $S'$ which is in turn in the interior of the convex hull of $S'.$  Further, the number of triangles $t$ in $S'$ is commensurate to the sum of degrees of $S;$ precisely
\(
t \geq \frac{1}{3}   \sum_{s \in S} \deg[\HPV](s) = \frac{1}{3} \VolV(S).
\)

On the one hand, Proposition~\ref{prop:strongarea} implies that the convex hull of the vertices of $S'\setminus S$ has area at least $c t$ for some constant $c$. On the other hand, by Lemma~\ref{lem:BE}, the area of the convex hull of $S' \setminus S$ is at most $4\pi' |S' \setminus S|.$  Combining these facts, we have that $|S' \setminus S| \geq \tfrac{c}{4\pi} \VolV(S).$

Since each vertex in $|S' \backslash S|$ is connected to $S$ by a boundary edge of $S$ in $\HPV$, we obtain $|\partial S| \geq |S' \setminus S|,$ and hence
\[
|\partial S| \geq |S' \setminus S|
\geq \tfrac{c}{4\pi} \VolV(S).
\]
This proves Theorem~\ref{thm:hpv} for $\HPV$.

\vspace{\baselineskip}

\noindent \emph{Proof that $\HPD$ has positive anchored expansion.}
Suppose once more that for some $k > k_0$ the vertices $d_1,\ldots,d_k \in \HPD$ form a connected subgraph of $\HPD,$ one of which is a triangle containing $0$. Let $t_1, \ldots, t_k$ denote the corresponding strongly connected Delaunay triangles in $\Htwo$.  If the subcomplex $S$ of the Delaunay triangulation with triangles $t_1, t_2, \ldots, t_k$ is not simply connected, let $S'$ be the smallest simply connected subcomplex of the Delaunay triangulation which contains $S.$  Then $\partial S \subset \partial S'$  and $\VolD(S) \leq \VolD(S').$  Hence $|\partial S|/\VolD(S) \geq |\partial S'| / \VolD(S'),$ and it suffices to assume that $S$ is simply connected.

Let $\X \subset \Htwo$ denote the set of vertices of the triangles $t_1,\ldots,t_k$, and let $\X' \subset \X$ be the set of points in $\X$ that are contained in the boundary of $\bigcup_{i=1}^k t_i$.  Observe that $\X$ and $\X'$ have the same convex hull, as an interior point of a set is in the interior of the convex hull of that set.

On the one hand, Proposition~\ref{prop:strongarea} implies that the convex hull of $\X$ has area at least $c\left|\X\right|$ for some constant  $c$.  On the other hand, by Lemma~\ref{lem:BE}, the convex hull of $\X'$ has area at most $4\pi|\X'|$ for any set of points $\X'\subset \Htwo$  Combined, these two inequalities give
\[
\frac{|\X'|}{|\X|} > \frac{c}{4\pi}.
\] 

As the boundary of the polygon $\bigcup_i t_i$ is a closed loop, there is a bijective correspondence between $\X'$ and boundary edges of $\bigcup_i t_i$. Further, the boundary edges of $\bigcup_i t_i$ are in bijective correspondence with the boundary edges of $S$ in $\HPD$, so $|\X'| = |\partial S|.$  

On the other hand, each vertex $d_i \in \HPD$ has degree $3.$  Thus we have that $|S| = 3k$.  As the embedding of $S$ gives a planar drawing of the graph, we have that by Euler's formula,
$e - k =  |\X| - 1,$ where $e$ is the number of edges in the complex $S.$  As every triangle contains exactly three edges and every edge is contained in at most $2$ triangles, $e \geq \tfrac 32 k,$ so that $|\X| \geq \tfrac 12 k = \tfrac 16 |S|.$  

Combining everything we have that
\[
  \frac{|\partial S|}{|S|}
  =
  \frac{|\X'|}{|S|}
  \geq
  \frac{|\X'|}{6|\X|}
  \geq
  \frac{c}{24\pi}.
\]

\end{proof}

\subsection{Preliminaries}
\label{sec:prelim}

The bulk of the work is to prove Proposition~\ref{prop:strongarea}, to which we devote the remainder of Section~\ref{sec:hpv}.  We will begin by setting some notation and frequently used identities. 
By an \emph{ordering} on a finite set $S$, we mean a bijection $\pi_S: \{1,2,\ldots,|S|\} \rightarrow S$.  We will often express an ordered set $S$ by listing its elements, using subscripts to indicate the elements' respective preimages under $\pi_S$.  
The hyperbolic distance between two points $x$ and $y$ is denoted as $\dH$, and the hyperbolic area of a region $R$ is denoted as $\VolH(R)$.  The hyperbolic disk centered at $x$ and with hyperbolic radius $r$ is denoted as $\HB(x,r)$.  We will frequently use the identity that
\begin{equation}
  \VolH(\HB(x,r)) = 2\pi(\cosh(r) - 1) \leq \pi e^r.
  \label{eq:ballvol}
\end{equation}
The circumference of the same ball is given by $2\pi \sinh(r).$

For three points $x,y,$ and $z$, the hyperbolic triangle they form is denoted by $\DeltaH$, and the angle at $x$ is denoted by $\angleH yxz$.  For three points $\{x,y,z\} \subset \Htwo,$ we let $\CDisc(x,y,z)$ denote the hyperbolic circumdisk through these points, if it exists.  As hyperbolic circles and Euclidean circles coincide in they Poincar\'e disk model, the hyperbolic circumcircle of these points in the disk model is just the Euclidean circumcircle, and the hyperbolic circumdisk exists and is finite if and only if the corresponding Euclidean circumcircle is contained in the unit disk.  We also use the notation $\CCtr(x,y,z)$ to be the hyperbolic center of this circumdisk, again if it exists.
Euclidean distances, areas, disks, triangles and angles are denoted the same way, except with the $\Htwo$ replaced by $\Etwo$.

We now give the proof of Lemma~\ref{lem:triangletail}.
\begin{proof}[Proof of Lemma~\ref{lem:triangletail}]
Suppose that some Delaunay triangle with one vertex the origin is not contained in $\partial \HB(0,r)$.  Then the circumcircle of this triangle contains a circle with radius $r/2$ and center on $\partial \HB(0,r/2)$.  Since the circumcircle cannot contain any points in the Poisson process $\PPP$, we deduce that there exists a disk with radius $r/2$ and center on $\partial \HB(0,r/2)$ containing no points of $\PPP$.  

We will define a \emph{finite} set $S$ of points on $\partial \HB(0,r/2)$ such that any disk with radius $r/2$ and center on $\partial \HB(0,r/2)$ contains $\HB(s,r/4)$ for some $s \in S$.  Let $S$ be a collection such that the hyperbolic distances between all neighboring pairs of points is exactly $r/4$, except possibly for one pair of neighboring points whose pairwise hyperbolic distance may be less than $r/4$.  Then any disk centered at a point on $\partial \HB(0,r/2)$ is at most distance $r/8$ from a point in $S$ and so contains $\HB(s,r/4)$ for some $s \in S.$  It remains to show that $S$ is not too large.

 To determine the size of $S$, we observe that, for a hyperbolic triangle $\DeltaH y0z$ with $\dH(0,y) = \dH(0,z) = r/2$ and $\dH(y,z) = r/4$, the angle $\alpha$ at $x$ satisfies
\[
\cos{\alpha} = \frac{\cosh^2\left(r/2\right) - \cosh{(r/4)}}{\sinh^2\left(r/2\right)} = 1 + \frac{1 - \cosh{( r/4)}}{\sinh^2\left(r/2\right)}
\]
by the law of cosines in $\Htwo$; so $\alpha > c e^{-3r/4}$ for some constant $c$.  Hence, \( |S|  < c' e^{3r/4} \) for some constant $c'$, and by a union bound, the probability that $\PPP \cap \HB(s,r/4)$ is empty for some $s \in S$ is at most
\(
c' e^{3r/4 - \lambda \pi e^{r/4}},
\)
completing the proof.
\end{proof}

For any $k \geq 3,$ define a \emph{triangulation scheme} to be a function $f : \{3,4,\ldots, k\} \to \binom{[k]}{2}$ with the properties:
\begin{enumerate}
  \item $f$ is injective on $\{4,\ldots,k\}$.
  \item $f(i) = \{f(i)_1, f(i)_2\} \in \binom{ [i-1] }{2}.$
  \item For all $j \in \{3,4,\ldots, k\},$ the edges $\{f(i)\}_{i=3}^{j}$ form a connected graph. 
  \item For every $j,$ the number of $i$ so that $j = \max(f(i))$ is at most $2.$
\end{enumerate}
If $\X = \{x_1,x_2,\ldots, x_k\}$ is an ordered collection of points in $\Htwo$, then, for each $i \in \{3,4,\ldots,k\}$, the vertices $x_i, x_{f(i)_1}$ and $x_{f(i)_2}$ define a closed hyperbolic triangle $\TFX$.  We will say that the pair $(\pi_{\X},f)$ is \emph{Delaunay} if every triangle $\TFX$ has a finite circumcircle, and \emph{planar} if all the triangles $\{\TFX\}_{i=3}^k$ have pairwise disjoint interiors.

A strongly connected collection of triangles whose union is additionally simply connected gives rise to a triangulation scheme:
\begin{lemma}
Let $\mathcal{T} $ be a strongly connected collection of triangles from $\HPD$, 
and let $t \in \mathcal{T}$. Denote by $\X$ the set of vertices of the triangles in $\mathcal{T}$. Almost surely, there is a triangulation scheme $f$ and an ordering $\{x_1,x_2,\ldots, x_k\}$ of $\X$ such that  $\TFX[3] = t$ and $\{\TFX\}_{i=3}^k \subseteq \mathcal{T}$. Note that the resulting pair $(\pi_{\X}, f)$ is both planar and Delaunay almost surely.
\label{lem:trischeme2}
\end{lemma}
\noindent This is proven is section~\ref{sec:schemes}.

Our strategy to prove Proposition~\ref{prop:strongarea} is to show that for any finite collection of points $\X \subset \PPP$ containing $0$, any ordering $\pi_{\X}$ of those points, and any triangulation scheme $f$ for which $(\pi_{\X},f)$ is planar and Delaunay, 
we have \[
\sum_{i=3}^{\ell} |\TFX| > c (\ell-2)
\]
provided $\ell = |\X|$ is sufficiently large.  By virtue of Lemma~\ref{lem:trischeme2}, this will give the desired area bound.
 
\subsection{Properties of triangulation schemes}
\label{sec:schemes}

We split the proof of Lemma~\ref{lem:trischeme2} into two parts, the first of which is the following lemma.
\begin{lemma}
Let $\mathcal{T}$ be a strongly connected collection of $\ell$ triangles from $\HPD$ whose union is simply connected in $\mathbb{H}$.  Also, let $t_*$ be any triangle in $\mathcal{T}$, and let  $V$ be any subcollection of the set of boundary edges of the triangulated polygon formed by the triangles in $\mathcal{T}$.  Then there is an ordering $\{t_1,t_2,\ldots,t_\ell\}$ of the triangles in $\mathcal{T}$ with $t_1 = t_*$ so that the following two properties are satisfied:
		\begin{enumerate}
			\item
For each $1 \leq i \leq \ell$, the triangles $\{t_1,t_2, \ldots, t_i\}$ are strongly connected, and their union is simply connected in $\mathbb{H}$.
			\item
Let $J$ be the set of all $1 < i \leq \ell$ for which $t_i$ shares exactly one edge with $t_1 \cup t_2 \cup \cdots \cup t_{i-1}$.  Denote this shared edge by $e_i$.  Let \[ \mathcal{S} = \left\{ e_i : i \in J \right\} \] Then $\mathcal{S} \cup V$ is connected.
		\end{enumerate}
\label{lem:trischeme1}
\end{lemma}

\begin{proof}

We proceed by induction on $\ell$.  For the base case, note that when $\ell = 1$, the result is satisfied since $t_1$ is connected and all $6$ subcollections of its edges are connected.  

Suppose that $\mathcal{T}$ is a strongly connected collection of triangles with $|\mathcal{T}| \geq 2$ from $\HPD$ and $V$ is any collection of boundary edges.  Let $P$ be the hyperbolic polygon given by the union of triangles in $\mathcal{T}.$  Suppose there is a triangle $t \in \mathcal{T}$ of which two of its sides are boundary edges of $P.$  Assume that $t_1  \neq t;$ an analogous argument covers the case that $t_* = t.$  Let $\mathcal{T}' = \mathcal{T} \setminus \{t\}.$  As $t$ only shares one edge with the rest of $\mathcal{T},$ it must be a leaf in any spanning tree in the dual graph on $\mathcal{T}.$  Hence, $\mathcal{T}'$ is strongly connected.  Moreover, since $t$ intersects the boundary of $P$, the triangles in $\mathcal{T}'$ also form a simply connected region in $\Htwo$.  Let $e$ be the interior edge of $t,$ and define $V'$ as the union of the edges of $V$ that are edges of some triangle in $\mathcal{T}'$ with $\{e\}.$  Apply the induction hypothesis to $(\mathcal{T}', V')$ to order the triangles in $\mathcal{T}'$  so that the properties listed in the lemma are satisfied. Extend this ordering to $\mathcal{T}$ by setting $t_{\ell} = t$.  It is easy to see that the ordering of $\mathcal{T}$ just defined satisfies the properties listed in the lemma.

Next, suppose that every triangle with an edge in the boundary of $P$ has exactly one such edge.  Let $\gamma$ be any closed loop in the boundary of the polygon $P,$ so that $\gamma$ is a union of boundary edges.  Let $e$ be a boundary edge that is not an isolated edge in $\gamma \cap V;$ there will always be at least $2$ such edges.  Thus we may choose the edge $e$ which is contained in a boundary triangle $t \neq t_*.$ Let $\mathcal{T}' = \mathcal{T} \setminus \{t \}.$  

There are two cases to consider.

\begin{itemize}
\item
  Suppose first that $\mathcal{T}'$ is strongly connected.  Then, set $V' = V \setminus \{e \}$ and apply the induction hypothesis.  As before, we extend the ordering of $\mathcal{T}'$ to one of $\mathcal{T}$ by setting $t_{\ell} = t$.  Property (1) of the lemma easily holds, so we just need to check that $\mathcal{T}$ satisfies property (2).  The set $\mathcal{S}$ defined in the lemma is the same for  $\mathcal{T}$ and  $\mathcal{T}'$; hence,  $\mathcal{S} \cup V'$ is connected.  Since $e$ is not an isolated edge in $\gamma \cap V$, we have that $\mathcal{S} \cup V$ is connected as well.
\item
Now, consider the case in which $\mathcal{T}'$ is not strongly connected.  In this case $\mathcal{T}$ can be decomposed as $\mathcal{T}_1 \cup \mathcal{T}_2' \cup \{t\}$ so that $\mathcal{T}_1$ and $\mathcal{T}_2'$ are both strongly connected collections of triangles, but the two collections are disjoint in $\HPD$.  Without loss of generality, let $\mathcal{T}_1$ be the component containing $t_*$.  Define $\mathcal{T}_2 = \mathcal{T}_2' \cup \{t\}.$  Let $g$ be the edge shared by $t$ and $\mathcal{T}_1,$ and $h$  the edge shared by $t$ and $\mathcal{T}_2'.$  Finally, set $V_1 = (V \cap \mathcal{T}_1) \cup \{g\}$ and  $V_2 = (V \cap \mathcal{T}_2 ) \cup \{g\}.$  

Apply the induction hypothesis to $(\mathcal{T}_1, V_1)$, setting $t_1  =t_* $.  Also apply the induction hypothesis to $(\mathcal{T}_2,V_2)$, setting $t_1 = t$.  Now, define an ordering of all the triangles in $\mathcal{T}$ by concatenating the ordered lists of triangles in $\mathcal{T}_1$ and $\mathcal{T}_2$, with the triangles in $\mathcal{T}_1$ coming first.  It is elementary to check that the resulting ordering satisfies property (1) of the lemma.  We turn to showing that $\mathcal{S} \cup V$ is connected; to prove this, it suffices to check that $g \in \mathcal{S}.$  Neither $e$ nor $h$ are edges in $\mathcal{T}_1.$  Hence, when the triangle $t$ appears in the ordering, the edge $g$ is added to $\mathcal{S}$.
\end{itemize}
This completes the induction.
\end{proof}

\begin{proof}[Proof of Lemma~\ref{lem:trischeme2}]

Apply Lemma~\ref{lem:trischeme1} with $V=\emptyset$ to obtain an ordering $\mathcal{T} = \{t_1,t_2,\ldots,t_\ell\}$ of the triangles of $\mathcal{T}$ with $t_1 = t$, so that the ordering satisfies the properties listed in the lemma.   Order the elements of $J$ as $i_1 < \cdots < i_r$, and define the mapping from $\mathcal{S}$ to $\X$ that sends each $e_{i}$ to the unique vertex of the triangle $t_{i}$ not contained in $e_i$.  This mapping is easily seen to be injective from the definition of $\mathcal{S}$.  Furthermore, every vertex except those in $t_1$ is in the range of this mapping. Hence, $k = r + 3 = |\mathcal{S}| + 3$.

Assign an ordering $\X = \{x_1,x_2,\ldots, x_k\}$ to the vertices of the triangles in $\mathcal{T}$ so that $t_1 = \{x_1,x_2,x_3\}$ and, for $3 < j \leq k$, the point $x_{j}$ is the unique vertex of triangle $t_{i_{j-3}}$ not contained in $e_{i_{j-3}}$.  Then, setting $f(3) = \{1,2\}$ and  $f(j) = e_{i_{j-3}}$ for $j > 3$, we have both $\TFX[3] = t_1$ and $\{\TFX\}_{i=3}^k \subseteq \mathcal{T}$.  It remains to prove that $f$ is a triangulation scheme. Conditions (1) and (2) for a triangulation scheme are easily satisfied, and condition (3) follows immediately from Lemma~\ref{lem:trischeme1}.  So we just need to check that $f$ satisfies condition (4).  

Suppose for contradiction that $f$ does not satisfy condition (4).  Then, for some $j$, we can find $p_1 < p_2 < p_3$ in $[k]$ with $\max(f(p_n)) = j$ for $n=1,2,3$.  Set $q_n = \min(f(p_n))$.  Since $t_1 \cup \bigcup_{2 \leq m < i_{j-3}} t_{m}$ is path-connected, we can find paths $\gamma_1,\gamma_2,\gamma_3$ in $t_1 \cup \bigcup_{2 \leq m <  i_{j-3}} t_{m}$ with endpoints $\{q_2,q_3\}$, $\{q_1,q_3\}$ and $\{q_1,q_2\}$, respectively.  Let $\gamma_1'$ be the closed loop in $t_1 \cup \bigcup_{2 \leq m < i_{p_3-3}} t_{m}$ obtained from $\gamma_1$ by adjoining to $\gamma_1$ the edges connecting $x_j$ to $x_{q_2}$ and $x_{q_3}$. Define $\gamma_2'$ and $\gamma_3'$ analogously.  Then, at least one of the triangles $t_{i_{p_n-3}}$ must be contained in one of the three regions bounded by the closed curves $\gamma_n'$.  But this means that  $t_1 \cup \bigcup_{2 \leq m \leq  i_{j-3}} t_{m}$ is not simply connected, contradicting the result of Lemma~\ref{lem:trischeme1}.
\end{proof}

As we will take a union bound over all ways to build a triangulation scheme for a collection of points, we additionally estimate the number of such schemes.
\begin{lemma}
	There is a constant $C > 0$ so that for any (unordered) collection $\X$ of $k$ points in $\Htwo$, the number of orderings $\pi_{\X}$ of $\X$ and triangulation schemes $f$ for which $(\pi_{\X},f)$ is planar is at most $(Ck)^k.$
	\label{lem:triangulationschemes}
\end{lemma}
\begin{proof}
We begin by introducing some terminology.  A \emph{planar drawing} of a graph is a representation of the graph in the plane in which the vertices of the graph are mapped to distinct points in the plane, and edges of the graph are mapped to continuous paths connecting the corresponding pairs of vertices.  A planar drawing is called \emph{crossing-free} if there are no crossings between the paths in the plane representing the edges in the graph.  By \cite[Theorem 2]{ACNS}, there is a $C_1 > 0$ so that the number of crossing-free subgraphs of any planar drawing of a graph on $k$ vertices is at most $C_1^k.$	

The complete graph with vertices $\X$ has a natural planar drawing given by connecting each pair of points by its hyperbolic geodesic.  Every planar $(\pi_{\X},f)$ gives rise to an ordered tuple of triangles $\mathscr{S}_{(\pi_{\X},f)} = (\TFX[3],\TFX[4], \ldots, \TFX[k]).$  Note that the mapping $(\pi_{\X},f) \mapsto \mathscr{S}_{(\pi_{\X},f)}$ is $6$-to-$1$. (We may freely choose the ordering of the first three vertices of $\X$, but the rest is determined.)  

Furthermore, the union of triangles in $\mathscr{S}_{(\pi_{\X},f)}$ is a crossing-free subgraph of the natural planar drawing of the complete graph on $\X$.  By the theorem cited above, there are at most $C_1^k$ such subgraphs.  Now, it is possible that, for different planar pairs $(\pi'_{\X},f')$ and $(\pi''_{\X},f'')$, the unions of $\mathscr{S}_{(\pi'_{\X},f')}$  and $\mathscr{S}_{(\pi''_{\X},f'')}$  give the same crossing-free subgraph.  We claim that there is a constant $C_2 > 0$ so that the number of planar pairs $(\pi_{\X},f)$ that yield the same crossing-free subgraph is at most $C_2^k k!$.  This will complete the proof, since then $6 (C_1C_2)^k k^k$ is an upper bound for the number of orderings $\pi_{\X}$ and triangulation schemes $f$ for which $(\pi_{\X},f)$ is planar.

Any crossing-free subgraph produced by a planar pair $(\pi_{\X},f)$ is a planar drawing of contiguous hyperbolic polygons, not all of them triangles.  By Euler's formula, the number $t$ of triangles that appear in the subgraph is at most 
$ 1 + e - k,$
where $e$ is the number of edges in the subgraph.  As each edge came from a triangle in $\mathscr{S}_{(\pi_{\X},f)}$ , we must have $e \leq 3k$; so the total number of triangles is at most $2k+1.$  Thus, the number of ways of picking an ordered $k$-element collection of triangles from this graph is at most $\binom{2k+1}{k} k! \leq 4^k k!.$
\end{proof}

\subsection{Probabilistic estimates}
\label{sec:prob}
We begin by showing that for a fixed triangulation scheme, the sum of areas of its triangles stochastically dominates a certain tree-indexed product of uniform variables that we now describe.  For any $i \in \{2,3,\ldots, k\},$ let $g(i)$ denote $\max(f(i)).$  Define a directed graph $\GF$  on the vertex set $\{2,3,\ldots, k\}$ with edge set given by $\{(i,g(i)) : i \in \{3,4,\ldots, k\}\}.$  Since $f$ is a triangulation scheme, the in-degree of any vertex is at most $2$ and the out-degree of every vertex is $1.$  As this is a connected graph on $k-1$ vertices with $k-2$ edges, $\GF$ is a tree.  Further, the edges are directed in such a way that from every vertex there is a directed path to $2.$  

Let $\alpha > 0$ and $1 \geq \beta > 0$ be fixed and let $U_2,U_3,\ldots, U_k$ be a collection of i.i.d. $\Unif[0,1]$ variables.  Let $Z_2 = \beta U_2^{\alpha/2}$ and define inductively for $i > 2$
\begin{equation}
	Z_i = \beta U_i^{\alpha/2} Z_{g(i)}^{1/\alpha}.
	\label{eq:Zi}
\end{equation}

\begin{lemma}
  Let $x_1 = 0,$ and let $r,s \geq 1.$ Let $x_2$ be picked uniformly from $\HB(0,s)$ according to hyperbolic area measure, and independently pick i.i.d.\,points $x_3, \ldots, x_k$ uniformly from $\HB(0,r)$ according to hyperbolic area measure.  Fix a triangulation scheme $f,$ and let $\X = \{x_1,x_2,\ldots, x_k\}.$  For any $\alpha > 0$ there is a $\beta > 0$ so that with $Z_i$ as in \eqref{eq:Zi} and for all $r,s \geq 1,$
\[
	\Pr \left[
		\sum_{i=3}^k \VolH(\TFX[i]) \leq t
		\text{ and $(\pi_{\X},f)$ Delaunay}
	\right]
	\leq
	\frac{\Pr \left[
		\sum_{i=3}^k Z_i \leq t
	\right]}{
		\VolH(\HB(0,r))^{k-2}
	}.
\]
	\label{lem:domination}
\end{lemma}

The key geometric estimate we need for this lemma is the following:
\begin{proposition}
	Suppose that $r> 0$ is fixed.  Let $z$ be a point that is picked uniformly from the $\HB(0,r)$ according to hyperbolic area measure.  There is an absolute constant $C>0$ so that 
	\[
		\Pr \left[
			\left| \Delta(x,y,z) \right| \leq \theta
			\text{and $\CDisc(x,y,z)$ exists}
		\right] \leq \frac{C\theta}{\dH(x,y)|\HB(0,r)|}.
	\]
	\label{prop:geometry}
\end{proposition}
This proposition is the most involved ingredient in proving Theorem~\ref{thm:hpv}, and it is the portion of the argument that relies most heavily on the geometry of $\Htwo,$ and so we devote section~\ref{sec:geometry} to it.  Once this geometric argument is made, the remainder of this section uses mostly abstract probabilistic analysis to complete the proof of Proposition~\ref{prop:strongarea}.
\begin{proof}[Proof of Lemma~\ref{lem:domination}]

For a fixed triangulation scheme, define variables $Q_f(i)$ and $Y_f(i)$ for $i=2,3,\ldots,k$ by setting $Y_f(2)= Q_f(2) = \dH(x_1,x_2)^\alpha,$ setting
\[
	Q_f(i) = \inf
	\{
		\dH(x_{f(j)_1},x_{f(j)_2})^\alpha : j \in \{3,4,\ldots,k\}, i=\max(f(j))
	\},
\]
and setting
\[
	Y_f(i) = 
	\begin{cases}
		\VolH(\TFX[i]) \wedge Q_f(i) & \CDisc(x(i),x_{f(i)_1},x_{f(i)_2}) \text{ exists,}\\
		\infty & \text{else.}
	\end{cases}
\]
As $Y_f(i) \leq \VolH(\TFX[i])$ on the event that $(\X,f)$ is Delaunay, it suffices to show that there is a $\beta>0$ so that
\[
	\Pr \left[
		\sum_{i=3}^k Y_f(i) \leq Ct \text{ and } \sup_{3 \leq i \leq k} Y_f(i) < \infty
	\right]
	\leq
	\frac{\Pr \left[
		\sum_{i=3}^k Z_i \leq t
	\right]}{
		\VolH(B(0,r))^{k-2}
	}.
\]

For each $i \in \{3,4,\ldots,k\},$ let $\filt_i=\sigma(x_1,x_2,\ldots, x_i).$  Applying a union bound, we have that
\begin{align}
	\Pr\left[
		Y_f(i) \leq t \middle\vert \filt_{i-1}
	\right]
	&\leq
	\Pr\left[
		\VolH(\TFX[i]) \leq t \text{ and $\CDisc(x_i,x_{f(i)_1},x_{f(i)_2})$ exists} \middle\vert \filt_{i-1}
	\right] \nonumber \\
	&+\sum_{\substack{j \in [k], \\ i = g(j)}} 
	\Pr\left[
		\dH(x_{f(j)_1},x_{f(j)_2})^\alpha \leq t \middle\vert \filt_{i-1}
	\right].
	\label{eq:yf1}
\end{align}
To the first term in the bound, we apply Proposition~\ref{prop:geometry}.
For the distance terms, we have that
\[
	\sup_{y \in \HB(0,r)}
	\Pr\left[
		\dH(x_{i},y)^\alpha \leq t \middle\vert \filt_{i-1}
	\right] \leq \frac{\VolH(\HB(0,t^{1/\alpha}))}{\VolH(\HB(0,r))}.
\]
Further, as $f$ is a triangulation scheme, the number of $j$ for which $i = g(j)$ is at most $2.$
Thus for some absolute constant $C>0,$ \eqref{eq:yf1} becomes
\[
	\Pr\left[
		Y_f(i) \leq t \middle\vert \filt_{i-1}
	\right]
	\leq 
	2\frac{\VolH(\HB(0,t^{1/\alpha}))}{\VolH(\HB(0,r))}
	+ \frac{Ct}{\dH(f(i)_1,f(i)_2)\VolH(\HB(0,r))}.
\]

For $t$ on compact sets, we have that $\VolH(\HB(0,t^{1/\alpha}))t^{-2/\alpha}$ stays bounded.
Furthermore, on the event that $Y_f(g(i)) < \infty,$ we have that $Y_f(g(i)) \leq \pi.$  Hence, 
\[
	\VolH(\HB(0,Y_f(g(i))^{1/\alpha^2}t^{1/\alpha}))t^{-2/\alpha}
	\one[Y_f(g(i)) <\infty]
\]
stays bounded for compact sets of $t.$  Also on the event that $Y_f(g(i)) < \infty,$ we have that $Y_f(g(i))^{1/\alpha} \leq \dH(f(i)_1,f(i)_2).$ Therefore, for all $t \leq \pi,$ we have
\begin{align*}
	\VolH(\HB(0,r))
	\Pr\left[
		Y_f(i) \leq t(Y_f(g(i)))^{1/\alpha} \middle\vert \filt_{i-1} 
	\right] 
	\one[Y_f(g(i)) <\infty]
	\hspace{-3.5in}&\hspace{3.5in} \\
	&\leq 
	\left(2\VolH(\HB(0,Y_f(g(i))^{1/\alpha^2}t^{1/\alpha}))
	+ \frac{CtY_f(g(i))}{\dH(f(i)_1,f(i)_2)}\right)
	\one[Y_f(g(i)) <\infty]
	\\
	&\leq 
	Ct^{2/\alpha}.
\end{align*}
for some constant $C>0.$  

Adjusting constants, we have that there is an absolute constant $\beta_1 > 0$ so that for $i > 2,$
\[
	\Pr\left[
		Y_f(i) \leq t\beta_1(Y_f(g(i)))^{1/\alpha} \middle\vert \filt_{i-1}
	\right]\one[Y_f(g(i)) <\infty] \leq
	\frac{
	\Pr\left[
		U_i^{\alpha/2} \leq t
	\right]}
	{
	\VolH(\HB(0,r))
	}.
\]
For $i =2,$ $Y_f(2) = \dH(x_{2},x_{1})^\alpha$ satisfies
\[
	\Pr\left[
	  \dH(x_{2},x_{1})^\alpha \leq t
	\right] \leq \frac{\VolH(\HB(0,t^{1/\alpha}))}{\VolH(\HB(0,r))}.
\]
Hence as $r \geq 1,$ we can find an appropriate constant $\beta_2>0$ so that
\[
	\Pr\left[
	  \dH(x_{2},x_{1})^\alpha \leq \beta_2 t
	\right] \leq t^{2/\alpha}.
\]
Setting $\beta=\beta_1 \wedge \beta_2,$ we now have for all $i \geq 2,$
\[
	\Pr\left[
		Y_f(i) \leq t\beta(Y_f(g(i)))^{1/\alpha} \middle\vert \filt_{i-1}
	\right]\one[Y_f(g(i)) <\infty] \leq
	\frac{
	\Pr\left[
		U_i^{\alpha/2} \leq t
	\right]}
	{
	\VolH(\HB(0,r))
	}.
\]
The result now follows by a standard induction argument.

\end{proof}

\begin{lemma}
  For any $\alpha > 2, \beta > 0,$ and $M >0$ there is an $\epsilon(\alpha,\beta, M)>0$ so that for all $k,$
  \[
	\Pr \left[
		\sum_{i=3}^k Z_i \leq \epsilon (k-2)
	\right] \leq e^{-M(k-2)}.
  \]
  \label{lem:Ztail}
\end{lemma}
\begin{proof}
	Suppose that for some $\tau > 0,$ 
	\[
		\sum_{i=3}^k \log(1/Z_i) < \tau (k-2).
	\]
	Then the number of $i$ so that $\log(1/Z_i) > 2\tau(k-2)$ is at most $\tfrac 12 (k-2).$  Hence for at least $\tfrac 12 (k-2)$ many $i,$ $Z_i \geq e^{-2\tau(k-2)}$ implying the existence of the $\epsilon$ we desire.  Thus, it suffices to show that there is a $\tau(\alpha,\beta,M) > 0$ so that
	\[
		\Pr \left[
			\sum_{i=3}^k \log(1/Z_i) \geq \tau (k-2)
		\right] \leq e^{-M(k-2)}.
	\]

	Recall \eqref{eq:Zi}, which states that $Z_i = \beta U_i^{\alpha/2} Z_{g(i)}^{1/\alpha}.$  Let $\mathcal{P}_i$ be the set of vertices in the unique path $\GF$ connecting $i$ to $2,$ and let $d(x,y)$ be the graph distance between vertices $x$ and $y$ in $\GF.$  Then, we can write
	\[
	  Z_i = \prod_{j \in \mathcal{P}_i} (\beta U_j^{\alpha/2})^{\alpha^{-d(j,i)}},
	\]
so that
	\begin{align*}
	  \log(1/Z_i) 
	  &\leq 
	  \sum_{j \in \mathcal{P}_i} -(\log \beta) \alpha^{-d(j,i)}
	  +\sum_{j \in \mathcal{P}_i} -(\log U_j) \alpha^{1-d(j,i)}/2 \\
	  &\leq 
	  \frac{-(\log \beta)\alpha}{\alpha -1}
	  +\sum_{j \in \mathcal{P}_i} -(\log U_j) \alpha^{1-d(j,i)}/2.
	\end{align*}
	Therefore, we can express the sum of $\log(1/Z_i)$ as
	\[
	  \sum_{i=3}^k \log(1/Z_i)
	  \leq (k-3)C_{\alpha,\beta}
	  + \sum_{j=2}^k f(j,\alpha) \log(1/U_j)
	\]
	where $C_{\alpha,\beta} = \frac{-(\log \beta)\alpha}{\alpha - 1}$ and $f(j,\alpha)$ is some coefficient.  Since $\GF$ has maximum in-degree $2,$ we have the uniform upper bound
	\[
	  f(j,\alpha)
	  \leq \sum_{\ell=0}^\infty 2^\ell (\alpha/2) \alpha^{-\ell}
	  =\frac{\alpha/2}{1-2/\alpha},
	\]
	as the number of vertices $m$ for which there is a directed path from $m$ to $j$ of length $\ell$ is at most $2^{\ell}.$

	Thus, we have that 
	\[
	  \sum_{i=3}^k \log(1/Z_i)
	  \leq (k-3)C_{\alpha,\beta} + D_\alpha \sum_{j=2}^k \log(1/U_j)
	\]
	for some positive constants $C_{\alpha,\beta}$ and $D_\alpha.$  Recall that $\log(1/U_j)$ is distributed as an $\Exponential(1)$ variable, which has some finite exponential moments.  Thus it follows that there is $\tau'(M) >0$ so that for all $M>0$ and all $k \geq 2,$
	\[
	  \Pr\left[
		\sum_{j=2}^k \log(1/U_j) > \tau'(M)(k-1)
	  \right] \leq e^{-M(k-1)},
	\]
and therefore it is possible to choose $\tau$ appropriately.

\end{proof}

\subsection{Core geometric estimate: the proof of Proposition~\ref{prop:geometry}}
\label{sec:geometry}
 



In this section we will heavily rely on the identification of $\Htwo$ with the open unit disk $D$ in $\mathbb{C}$ via the Poincar\'e disk model.  By applying an isometry, we can assume without loss of generality that $y = 0$ and that $x$ lies on the positive real axis.  Let $S$ denote the set of points $z \in \Htwo$ such that $\CDisc(0,x,z)$ exists, and let $A$ be the set of  $z \in \Htwo$ for which $\VolH (\Delta(0,x,z)) \leq \theta$.  
In order to prove Proposition~\ref{prop:geometry}, it suffices to show that, under the conditions of the proposition, the hyperbolic area of $S \cap A$ is at most $ \frac{ C\theta}{x}$ for some constant $C$.

Having identified $\Htwo$ with the disk $D$, we can describe the regions $S$ and $A$ explicitly.  First, we describe the region $S$.  Let $G$ denote the hyperbolic geodesic ray from $x/2$ normal to the real axis and contained in the upper half-plane.  Let $q$ denote the limit point of $G$ on $\partial D$.  Then the horocycle $H$ through $0,x,$ and $q$ is precisely the limit  of the circles through $0, x$ and $p$ as $p$ approaches $q$ along the geodesic $G$. Moreover, the set of circumcenters $\CCtr(x,y,z)$ for $z \in S$ is exactly $G$.  Finally, let $F$ be the circle with diameter given by the segment from $0$ to $x$ in terms of which we can concisely describe $S$ as the closure of $H \setminus F.$

The region $A$ is characterized explicitly by the following lemma,  a special case of Theorem 7 of~\cite{KarpPeyerimhoff}.

\begin{lm}
For each $\alpha > 0$, the locus of points $y \in D$ in the upper half-plane with $\VolH(\Delta(0xy))= \alpha$ is given by the intersection of $D$ with the ray $\ell$ from $1/x$ that makes the angle $\alpha/2$ with the negative real axis.
\end{lm}

Having described the regions $S$ and $A$ explicitly, we are now ready to prove the proposition.  We divide our proof into two parts: one for $x$ less than or equal to some $\delta \in (0,1)$, and one for $x > \delta$.

In the case that $x \leq \delta,$ we will use estimates comparing hyperbolic and Euclidean volume.  Towards this end, we need the following elementary lemma.
\begin{lm}
Fix $0 < \delta < 1$.  For all $0 \leq x \leq \delta$ and $c < \frac{\sqrt{1 - \delta^2}}{2}$, we have the following:
\begin{enumerate}
\item
If $z \in H$ with $\Re{z} > 0$ and $\Im{z} \leq c$, then $\Re{z} \leq \frac{\delta}{2} + \sqrt{ c - c^2 }.$
\item
If $z \in H$ with $\Re{z} < 0$ and $\Im{z} \leq c$, then $\Re{z} \geq - \sqrt{ c - c^2 }.$
\end{enumerate}
	\label{lm:distancefromorigin}
\end{lm}

\begin{proof}
  Write $H_x$ for the circumcircle of $0,x$ and $q,$ i.e., the horocycle $H.$
Consider the region of $z \in \bigcup_{x' \in (0,\delta)} H_{x'}$ for which $\Im{z} < c$.  The point in this region with maximal real part lies on $H_{\delta}$; the point in this region with minimal real part lies on $H_{\delta}$.  Our restriction on $C$ implies that both these extremal points have imaginary part exactly $c$.  The result then follows from two applications of the Pythagorean theorem.
\end{proof}

 Observe that $\ell$ intersects the vertical line $\Re{(z)} = -1$ at the point $-1 + \left(1 + \frac{1}{x}\right) \tan{(\theta/2)} i$.  For all $x<1$ and all $\theta<\frac{\pi}{2}$, we have
\[
\left(1 + \frac{1}{x}\right) \tan{(\theta/2)} \leq \frac{2\theta}{x}.
\]
Thus, all points in $S \cap A$ have imaginary part at most $\frac{2\theta}{x}$.
It follows from Lemma~\ref{lm:distancefromorigin} that we can choose $0 <\delta$ and $c < 1$ sufficiently small that for $x \leq \delta$ and $\frac{\theta}{x} < c$ the region $S \cap A$ is contained in the disk centered at the origin with Euclidean radius $\frac{\sqrt{2}}{2}$.  In this case,
\begin{align*}
\VolH(S \cap A) 
&= \int_{S \cap A} \frac{4r}{(1-r^2)^2} dr d\theta\\
&\leq 16 \int_{S \cap A} r dr d\theta\\
&= 16 \VolE(S \cap A).
\end{align*}
Since $S \cap A$ is contained in the rectangle with vertices  $\pm 1$ and $\pm 1 + i\frac{2\theta}{x}$, its Euclidean area is bounded above by $\frac{4\theta}{x}$.  

We deduce that Proposition~\ref{prop:geometry} holds (for some choice of the constant $C$) whenever both $x < \delta$ and $\frac{\theta}{x} < c$.  In fact, we can discard the latter condition by stipulating that $C > c$.  

Thus, we have reduced to the case $x > \delta$, where $\delta$ is the value chosen above.  As in the previous case, it suffices to prove the proposition for $\frac{\theta}{x} < c$ for any choice of constant $c>0$.  Observe that, for $x > \delta$, the condition $\frac{\theta}{x} < c$ holds as long as $\theta < c \delta$.  Hence, it is enough to prove the following claim:

\begin{claim}
Fix $\delta \in (0,1)$, and suppose $x > \delta$.  Then we can choose a constant $\theta_{\delta}$ and $C>0$ so that, for all $\theta < \theta_{\delta}$, the $\VolH(S \cap A)< C \theta.$
\label{claim:xdelta}
\end{claim}

We begin our proof of this claim with a diagram, Figure~\ref{fig:diagram}.  
As in the figure we can assume the ray $\ell$ intersects $F$ at two distinct points by choosing $\theta_{\delta}$ sufficiently small.  Label these points as $w_0$ and $w_1$ where $w_0$ is the first intersection of $\ell$ with $F$.  Also let $w_2$ be the second intersection of $\ell$ with $H.$ Define $R_1$ to be the set of $z \in S \cap A$ with $\Re z \leq \Re w_1,$ define $R_2$ to be the set of $z \in S \cap A$ with $\Re z \geq \Re w_0$ and define $R_3$ as the set of $z \in S \setminus R_2$ below the Euclidean line $0$ and $w_0.$  





\begin{figure}
\begin{tikzpicture}
\clip (-2,-2) rectangle (9,5);
\draw (0,0) circle [radius=4];
\draw [thick, <->] (-5,0) -- (9,0);
\draw (1,0) circle [radius=1];
\draw (1,1.7320508075) circle [radius=2];
\node [above right] at (2,3.464) {$q$};
\draw[fill] (0,0) circle [radius=0.05];
\draw[fill] (1,1.732) circle [radius=0.05];
\draw[fill] (1,0) circle [radius=0.05];
\draw[fill] (2,0) circle [radius=0.05];
\node [below left] at (0,0) {0};
\node [above] at (1,-.12) {$x/2$};
\node [below right] at (2,0) {$x$};
\node [below right] at (2.828428,-2.828428) {$D$};
\draw [thick][->] (8,0) -- (-1.8,9.8/8);
\node [above left] at (5,-0.09295) {$\theta/2$};

\draw (0,0) --  (1.871 * 1.59636,1.871 * 0.802713);

\definecolor{redlight}{RGB}{250,150,150}
\definecolor{greendark}{RGB}{0,150,0}
\definecolor{bluelight}{RGB}{0,0,250}

\path [pattern=crosshatch dots,pattern color=blue] (1.59636,0.802713) arc (53.39:0:1) arc (-60:-32.25:2) -- cycle;
\path [pattern=crosshatch dots,pattern color=greendark] (1.59636,0.802713) --  (1.871 * 1.59636,1.871 * 0.802713) arc (-6.81: -32.25:2) -- cycle;
\path [pattern=crosshatch dots,pattern color=red] (0.619576,0.924812) arc (112.36:180:1) arc (-120:-162:2) -- cycle;

\node [below left] at (-.6,.8) {$R_1$};
\node [below right] at (2.4,.6) {$R_2$};
\node [right] at (2.9,1.0) {$R_3$};
\node [below left] at (1.93,-.807) {$F$};
\draw[fill] (2,2 * 1.732) circle [radius=0.05];

\draw[fill] (8,0) circle [radius=0.05];
\node [below] at (8,0) {$1/x$};

\draw[fill] (.619576,0.924812) circle [radius=0.05];
\draw[fill] (1.59636,0.802713) circle [radius=0.05];
\draw[fill] (-0.898,1.11402) circle [radius=0.05];
\node [above right] at (-0.99,1) {$w_2$};
\node [above] at (.6,0.87) {$w_1$};
\node [above] at (1.59636,0.802713) {$w_0$};

\node [above] at (-1.8,9.8/8) {$\ell$};



\end{tikzpicture}
\caption{}
  \label{fig:diagram}
\end{figure}

Our strategy is to bound the hyperbolic area of $S \cap A = R_1 \cup R_2$ from above by the sum of the hyperbolic area of $R_1$ and the hyperbolic area of $R_2 \cup R_3$.
First, however, we derive a bound on $\varphi_* := |\angle x 0 w_0|$ in terms of $\theta$ and $x$. 
 By the law of sines on the Euclidean triangle with vertices $w_0, x/2$ and $1/x$,
\[
\frac{\sin{(\theta/2 + 2 \varphi_*)}}{\frac{1}{x} - \frac{x}{2}} = \frac{\sin{\theta}}{\frac{x}{2}}
\]
and therefore 
\[ 
x^2 \sin{(\theta/2 + 2 \varphi_*)} = \sin{(\theta/2)} (2-x^2).
\] 
Applying the sine addition formula to the first term, we can arrange terms to obtain the following quadratic in $\sin( 2 \varphi_*):$ 
\[
  x^4 \sin^2(\theta/2) ( 1 - \sin^2( 2 \varphi_* ) ) 
  =
  \left( x^2 \cos(\theta/2) \sin(\theta/2) - \sin(\theta/2)( 2-x^2)\right)^2.
\]
Applying the quadratic formula and taking the smaller solution, since the larger solution corresponds to $| \angle x 0 w_1 |,$ we get \[ \sin{(2\varphi_*)} = \sin{(\theta/2)} \left( z \cos{(\theta/2)} - \sqrt{ 1 - z^2 \sin^2{(\theta/2)}} \right), \] where $z = \frac{2-x^2}{x^2} $. Now,
\begin{align*}
\sin{(2\varphi_*)}
&=  \sin{(\theta/2)} \left( z \cos{(\theta/2)} - \sqrt{ 1 - z^2 \sin^2{(\theta/2)}} \right)\\
&\leq \sin{(\theta/2)} \left( z \cos{(\theta/2)} - 1 + z^2 \sin^2{(\theta/2)} \right)\\
&= \sin{(\theta/2)} \cos{(\theta/2)} (z-1) + \sin{(\theta/2)} (\cos{(\theta/2)} - 1) + z^2 \sin^3{(\theta/2)}\\
 &\leq \sin{(\theta/2)} \cos{(\theta/2)} (z-1) + \sin{(\theta/2)} (-\sin^2{(\theta/2)}) + z^2 \sin^3{(\theta/2)}\\
&= \frac{1}{2} \sin{\theta} (z-1) + \sin^3{(\theta/2)} (z^2-1) \\
&\leq C_{\delta}'  \theta (1-x), \qquad \text{for $x > \delta.$}
\end{align*}
It follows that, for $x > \delta$, \begin{equation} \sin{(2\varphi_*)} \leq  C_{\delta}'  \theta (1-x) \label{phibound1} \end{equation}
for some constant $C_{\delta}'$ depending on $\delta$.   Hence, we can choose $\theta_{\delta}$ sufficiently small so that  \begin{equation} \sin{(2\varphi_*)} \leq \frac{2}{3} (1-x) \label{phibound} \end{equation} and  \begin{equation} \varphi_* \leq C_{\delta}'  \theta (1-x) \label{phibound2}. \end{equation}

\subsubsection*{The region $R_2 \cup R_3$.}

For each $0 \leq \varphi \leq \varphi_*$, consider the line through the origin that makes the angle $\varphi$ with the positive real axis. Let $p_1(\varphi)$ and $p_2(\varphi)$ be the points at which this line intersects the circle $F$ and the horocycle $H$, respectively.  Set $\ell_1(\varphi) = \left| p_1(\varphi) \right|$ and $\ell_2(\varphi) = \left| p_2(\varphi) \right|$.  Then the hyperbolic area of $R_2 \cup R_3$ is given by 
\begin{align*}
\VolH{(R_2 \cup R_3)} 
&= \int_0^{\varphi_*} \left(  \frac{2}{1 - (\ell_2(\varphi))^2} - \frac{2}{1 - (\ell_1(\varphi))^2} \right)  d\varphi\\
&= \int_0^{\varphi_*} \frac{2 \left((\ell_2(\varphi))^2 - (\ell_1(\varphi))^2 \right) }{ \left(1 - (\ell_1(\varphi))^2 \right) \left(1 - (\ell_2(\varphi))^2 \right)} d\varphi\\
&= \int_0^{\varphi_*} \frac{2 \left(\ell_2(\varphi) - \ell_1(\varphi) \right) }{ \left(1 - \ell_1(\varphi) \right) \left(1 - \ell_2(\varphi)\right)}  \frac{\ell_2(\varphi)+ \ell_1(\varphi) }{ \left(1 + \ell_1(\varphi) \right) \left(1 + \ell_2(\varphi)\right)} d\varphi\\
&= \int_0^{\varphi_*} \frac{4 \left(\ell_2(\varphi) - \ell_1(\varphi) \right) }{ \left(1 - \ell_1(\varphi) \right) \left(1 - \ell_2(\varphi)\right)} d\varphi.
\end{align*}
The following explicit formulas for $\ell_1$ and $\ell_2$ are easily verified:
\begin{align*}
\ell_1(\varphi) &= x \cos{\varphi},\\
\ell_2(\varphi) &= x \cos{\varphi} + \sqrt{1 - x^2} \sin{\varphi}.
\end{align*}
Hence
\[
\VolH{(R_2 \cup R_3)}
=
4 \int_0^{\varphi_*} \frac{ \sqrt{1-x^2} \sin{\varphi} d\varphi}{\left( 1-x \cos{\varphi}\right) \left( 1 - x \cos{\varphi} - \sqrt{1-x^2} \sin{\varphi} \right)}.
\]
Substituting $u = \cos{\varphi}$, we can rewrite this integral as
\[
\VolH{(R_2 \cup R_3)}
=
4 \int_{\cos{\varphi_*}}^{1} \frac{ \sqrt{1-x^2} du}{\left( 1-xu \right) \left( 1 - x u- \sqrt{1-x^2} \sqrt{1-u^2}  \right)}.
\]
Rationalizing the denominator, we get
\[
\VolH{(R_2 \cup R_3)}
=
4 \int_{\cos{\varphi_*}}^{1} \frac{ \sqrt{1-x^2} \left( 1 - xu + \sqrt{1-x^2}\sqrt{1-u^2} \right) du}{\left( 1-xu \right) \left(x-u \right)^2}.
\]
Since $\sqrt{1-x^2}\sqrt{1-u^2} \leq 1-xu$, we get that
\begin{align}
\nonumber
\VolH{(R_2 \cup R_3)}
&\leq\sqrt{1-x^2} \int_{\cos{\varphi_*}}^{1} \frac{ du}{\left(x-u \right)^2}\\
\nonumber&= 8 \sqrt{1-x^2}\left[ \frac{-1}{1-x} + \frac{1}{\cos{\varphi_*} - x} \right]\\
&= 8 \sqrt{1-x^2} \frac{1-\cos{\varphi_*}}{(1-x)\left(\cos{\varphi_*} - x \right)}. \label{lastline}
\end{align}
By (\ref{phibound}) above, we have 
\[ 
1 - \cos{\varphi_*} \leq \sin{(2\varphi_*)} \leq \frac{2}{3} (1-x),
\] 
and therefore 
\[ 
(\cos{\varphi_*} - x) = (\cos{\varphi_*} - 1) + (1-x) \geq \frac{1}{3} (1-x). 
\] 
Hence (\ref{lastline}) is at most
\begin{align*}
 \frac{24 \left(1-\cos{\varphi_*}\right)}{(1-x)^{3/2}}
&\leq  \frac{12 \varphi_* }{(1-x)^{3/2}}\\
&\leq 12 C_{\delta}'^2 \theta^2 (1-x^2)^{1/2} \qquad \text{by (\ref{phibound2})}\\
&\leq C_{\delta}'' \theta^2
\end{align*}
for some constant $C_{\delta}''$ depending on $\delta$.
We conclude that 
\begin{equation}
\VolH{(R_2 \cup R_3)} \leq C_{\delta}'' \theta^2. \label{R2R3} 
\end{equation}

\subsubsection*{The region $R_1$.}

Next, we consider the region $R_1$.  We claim that, for $\theta_{\delta}$ sufficiently small and $x > \delta$, the points $w_1$ and $w_2$ are contained in a ball of radius $1/\sqrt{2}$.  

First, if we let $t$ denote the unique point in the upper half-plane such that the line through $1/x$ and $t$ is tangent to $C$ at $t$, then \[\dE(0,w_1) \leq \dE(0,t) =  \frac{x}{\sqrt{2-x^2}}. \] Since the latter is monotonic and tends to $1/\sqrt{2}$ as $x \rightarrow 1$, we deduce that $w_1$ lies in the disk centered at the origin with radius $1/\sqrt{2}$.

Next, we can bound $\dE(0,w_2)$ from above by the length of the arc of the horocycle $H$ between $0$ and $w_2$. The latter is just $\pi \alpha$, where $\alpha$ is the angle of the arc.  By (\ref{phibound}), \[ 2 \theta \geq \alpha - \varphi_*  \geq \alpha - C_{\delta}' (1-\delta) \theta. \] Therefore, $\alpha$ is less than a constant (depending only on $\delta$) times $\theta$.  It follows that $\dE(0,w_2) \leq 1/\sqrt{2}$ for sufficiently small $\theta_{\delta}$. 

We conclude that, for $\theta_{\delta}$ sufficiently small and $x > \delta$, the points $w_1$ and $w_2$ are contained in a ball of radius $1/\sqrt{2}$, as claimed.  
Thus, 
\begin{align*}
\VolH{(R_1)} 
&= \int_{R_1} \frac{4r}{(1-r^2)^2} dr d\theta\\
&\leq 16 \int_{R_1} r dr d\theta\\
&= 16 \VolE{(R_1)}.
\end{align*}
Observe that $R_1$ is contained in the disk with center $1/x$ and radius $1/x + 1$.  In fact, $R_1$ is contained in the circular sector of this disk bounded by the ray $\ell$ and the real axis.  The latter sector has area $\frac{\theta}{4} \left( 1 + \frac{1}{x} \right)^2 $.  Hence, 
\begin{equation}
\VolH{(R_1)} \leq C_{\delta}''' \theta \label{R1} \end{equation}
where $C_{\delta}''' = \frac{1}{4} \left( 1 + \frac{1}{\delta} \right)^2$.

\vspace{5 mm}

Combining (\ref{R2R3}) and (\ref{R1}) proves Claim~\ref{claim:xdelta}, and therefore the proposition.

\subsection{Proof of Proposition \ref{prop:strongarea}}
\label{sec:final}
  Suppose that there is a constant $c > 0$ and a collection of Delaunay triangles $t_1,t_2, \ldots, t_k$ that are strongly connected to the origin and whose union $\cup_{i=1}^k t_i$ is simply connected for which
	\[
		\sum_{i=1}^k \VolH(t_i) \leq \tfrac12ck.
	\]
Let $\X$ be the vertex set of these triangles, and let $\ell = |\X|.$  As the complex formed by $t_1,t_2,\ldots, t_k$ is planar, we have from Euler's formula that $\ell \geq \tfrac 12 k.$
By Lemma~\ref{lem:trischeme2}, there is an ordering $\pi_\X$ of $\X$ and a triangulation scheme $f$ so that $(\pi_\X, f)$ is planar and Delaunay.  Further, we may take the ordering so that $x_1 = 0$ and $\{x_1,x_2,x_3\}$ are the vertices of $t_1 = \TFX[3].$  Finally, the condition on the sum of the area of triangles implies 
\begin{equation}
  \label{eq:contra}
  \sum_{i=3}^\ell \VolH(\TFX[i]) \leq \tfrac12ck \leq c\ell.
\end{equation}

This motivates the definition of the event $\mathcal{E}_{c,r,\ell}$ that there exists a triple $(\X,\pi_\X,f)$ where
\begin{itemize}
  \item $\X \subset \PPP \cap \HB(0,r)$ is a set containing $0$ with $|\X| = \ell,$
  \item $\pi_\X$ is an ordering of $\X$ putting $0$ first, and
  \item $f$ is a triangulation scheme
\end{itemize}
so that 
\begin{itemize}
  \item the pair $(\pi_\X, f)$ is planar and Delaunay,
  \item
\(
  \sum_{i=3}^\ell \VolH(\TFX[i]) \leq c\ell,
\) 
and
\item
  the diameter of $t_1$ is at most $\ell.$
\end{itemize}
To complete the proof, it suffices to show that there is a $c>0$ so that  
\(\Pr\left[\cup_{r=1}^\infty \mathcal{E}_{c,r,\ell}\right]\) is summable in $\ell.$  
Then by Borel-Cantelli, there is some random $\ell_0 < \infty$ larger than the diameter of all the Delaunay triangles incident to the origin so that for $\ell \geq \ell_0,$ 
\[
  \sum_{i=3}^\ell \VolH(\TFX[i]) \geq c\ell,
\]
from which follows Proposition~\ref{prop:strongarea}.

The following lemma therefore concludes the proof.
\begin{lemma}
  There are constants $c>0$ and $\delta$ so that for all $\ell \geq 3,$ 
  \[
    \Pr\left[
      \cup_{r=1}^\infty
      \mathcal{E}_{c,r,\ell}
    \right] \leq e^{-\delta \ell}.
  \]
  \label{lem:mule}
\end{lemma}
\begin{proof}
  Combining Lemmas~\ref{lem:domination} and \ref{lem:Ztail} it follows that for any $M > 0$ there is an $\epsilon(M) > 0$ so that with $\X = \{ 0 = x_1, x_2, x_3,\ldots, x_\ell \}$ like in the statement of Lemma~\ref{lem:domination} we have for any fixed triangulation scheme $f$
  \begin{equation}
   	\Pr \left[
		\sum_{i=3}^\ell \VolH(\TFX[i]) \leq \epsilon(\ell-2)
		\text{ and $(\pi_{\X},f)$ Delaunay}
	\right]
	\leq
	\frac{e^{-M(\ell-2)}}{
		\VolH(\HB(0,r))^{\ell-2}
	}.
    \label{eq:epsi}
  \end{equation}

  Let $N_r = | \PPP \cap \HB(0,r)|$ and $N_\ell = | \PPP \cap \HB(0,\ell) |.$  By Lemma~\ref{lem:triangulationschemes} we have that for any set $\X \subset \PPP$ of size $\ell$ there are at most $(C\ell)^\ell$ pairs $(\pi_\X,f)$ that are planar and Delaunay. Thus, we have for any $r \geq \ell$
  \begin{equation}
    \Pr\left[
      \mathcal{E}_{c,r,\ell} ~|~ N_r, N_\ell
    \right]
    \leq (N_\ell-1) \binom{N_r - 2}{\ell - 2} (C\ell)^\ell 
	\frac{e^{-M(\ell-2)}}{
		\VolH(\HB(0,r))^{\ell-2}
	}.
    \label{eq:condi}
  \end{equation}

  It is easily checked that for $\PPP,$
  \[
    \lim_{r \to \infty} \Exp \left[ \frac{N_\ell N_r^{\ell - 2}}{
    \VolH(\HB(0,r))^{\ell-2}}\right] = \lambda^{\ell - 1} \VolH(\HB(0,\ell)).
  \]
  Hence taking expectations and limits in~\eqref{eq:condi}, we get
  \begin{align*}
    \lim_{r \to \infty}
    \Pr\left[
	\mathcal{E}_{c,r,\ell}
    \right]
    &\leq
    \lim_{r \to \infty}
    \frac{ (C\ell)^{\ell}e^{-M(\ell-2)}}{(\ell -2)!}
    \Exp \left[ \frac{N_\ell N_r^{\ell - 2}}{
    \VolH(\HB(0,r))^{\ell-2}}\right] \\
    &\leq
    \frac{ (C\ell)^{\ell}e^{-M(\ell-2)}\lambda^{\ell-1} \VolH(\HB(0,\ell))}{(\ell -2)!}.
  \end{align*}
  By making $M$ sufficiently large, this can be made smaller than some $e^{-\delta \ell}$ for all $\ell\geq 3.$ Note that since the events $\mathcal{E}_{c,r,\ell}$ are nested, the lemma follows from monotone convergence. 

\end{proof}


\section{Miscellaneous properties of $\HPV$ and $\HPD$}
\label{sec:other}

We begin with:
\begin{proof}[Proof of Proposition~\ref{prop:stationarity}]
  We will recall the notation from the statement of the proposition.
  Let $\mathbb{M}$ be a Riemannian symmetric space.  Let $\mathcal{X}$ be a Poisson point process with positive intensity $\lambda \cdot dV$ where $\lambda > 0$ and $dV$ denotes the volume element and which is conditioned to have a point at some fixed $x \in \mathbb{M}$.  Let $\mathscr{G}$ be the dual graph of the Voronoi tessellation, so that two nuclei $x,y \in \mathcal{X}$ are adjacent if and only if there is an open metric ball $B$ with $\{x,y\} \subset \partial B$ and $B \cap \mathcal{X} = \emptyset.$  Let $\rho$ denote the vertex of $\mathscr{G}$ whose Voronoi cell has nucleus $x.$  We will let $P$ denote the law on $(\mathscr{G},\rho)$ that arises as the push forward of the point process distribution, and we will let $Q$ be the law on rooted graphs with Radon-Nikodym derivative $\frac{dQ}{dP} = \frac{\deg \rho}{\Exp \deg \rho}.$  We will continue to use $\Pr$ to denote the law of the Poisson point process and $\Exp$ to denote expectation with respect to $\Pr.$

For any two nonequal points $y,z \in \mathbb{M}$ there is a geodesic $\gamma$ connecting them.  On this geodesic, we may find the midpoint $m$ between $y$ and $z.$   As $\mathbb{M}$ is a symmetric space, there is an isometry $\tau_{y,z}$ of $\mathbb{M}$ that interchanges $y$ and $z$ and fixes $m.$  Let $\mathbb{B}_{y,z}$ be the bisector of $y$ and $z,$ i.e.\,the submanifold of $\mathbb{M}$ consisting of points that are equidistant from $y$ and $z,$ and let $\mathcal{C}_{y,z}$ be the event that $\left\{ x,y \right\} \subset \mathcal{X}$ and $y$ connects to $z$ in $\mathscr{G}.$  For any bijection $\tau : \mathbb{M} \to \mathbb{M},$ let $\tau^* : \sigma(\mathcal{X}) \to \sigma(\mathcal{X})$ be the induced map, i.e.
  \[
    \tau^*(\left\{ \mathcal{X} \in V \right\}) = 
    \left\{ \tau(\mathcal{X}) \in V \right\}
  \]
  where $V$ is any subset of $\mathbb{M}^{\mathbb{N}}.$  As $\mathbb{B}_{x,y}$ is invariant under $\tau_{x,y}$ it follows that $\tau_{x,y}^*( \mathcal{C}_{x,y} ) = \mathcal{C}_{x,y}.$

  We begin by showing that $\Exp \deg \rho < \infty.$  If $\mathbb{M}$ is compact, there is nothing to do, as the space has finite volume and hence $|\mathcal{X}| < \infty$ almost surely.  Therefore, it suffices to consider the case that $\mathbb{M}$ is non-compact, in which case it decomposes as a Riemannian product $\mathbb{M}_1 \times \mathbb{M}_2$ where $\mathbb{M}_1$ is non-positively curved and $\mathbb{M}_2$ is compact (see \cite[Chapter 4]{Helgason} for an overview).  We will assume for simplicity that $\mathbb{M}_2$ is trivial so that $\mathbb{M} = \mathbb{M}_1$ is non-positively curved.  In the general case, the proof can be adapted by noting the projection of $\mathbb{M}$ onto $\mathbb{M}_1$ is a quasi-isometry.
  
  Let $f(r)=\VolM(\MB(x,r)),$ noting that the definition is independent of $x.$  As $\mathbb{M}$ is non-positively curved, we have that $f(r) \geq cr^d$ for some constant $c>0$ where $d$ is the dimension of $\mathbb{M}.$  By a standard packing argument, it is possible to find a $1$-net of the ball $\MB(x,r)$ of cardinality at most $ f(r)/f(1/2).$  Hence, by packing $\MB(x,r-1),$ we have that all of $\MB(x,r)$ is contained in the union of $2$-balls centered at points in the net, so that 
  \begin{equation}
    f(r) \leq f(r-1)f(2)/f(1/2).
    \label{eq:volrecurrence}
  \end{equation}
Note this implies that $f(r)$ grows at most exponentially.

Let $\Pr_{y}$ be the law of $\mathcal{X}$ conditioned to have points at both $x$ and $y.$ To show that $\Exp \deg \rho < \infty$ it will suffice to show that
\begin{equation}
  \Pr_y[ \mathcal{C}_{x,y}] \leq C\exp( - f(\dM(x,y)/2)/C)
  \label{eq:connectiondecay}
\end{equation}
for some constant $C,$ as having shown this it follows that
\begin{align*}
  \Exp \deg \rho
  &= \Exp \sum_{y \in \mathcal{X}} \one[x,y \text{ connected in } \mathscr{G}] \\
  &= \lambda \int_{\mathbb{M}} \Pr_y(\mathcal{C}_{x,y})\,dV(y) \\
  &\leq C\lambda \int_0^\infty \exp(-f(r/2)/C)f'(r)\,dr \\
  &<\infty.
\end{align*}

Turning to the proof of \eqref{eq:connectiondecay}, recall that $x$ and $y$ are connected if and only if there is some $u \in \mathbb{B}_{x,y}$ so that $\MB(u,\dM(u,x))$ contains no points of $\mathcal{X}.$  As balls are geodesically convex in a complete non-positively curved space (see \cite[Section 1.6]{Eberlein}), this ball also contains the midpoint $m$ of $x$ and $y.$  Hence $\dM(u,x) \geq \dM(u,m).$  Additionally, the bisector is disjoint from $\MB(x,\dM(x,m))$ and so $\dM(u,x) \geq \dM(x,m).$

Let $\mathcal{U}$ be a maximal $1$-separated subset of $\mathbb{B}_{x,y}.$  Note that if $\mathcal{C}_{x,y}$ occurs, then there is some $u \in \mathcal{U}$ so that $\MB(u,\dM(u,x)-1)$ contains no points of $\mathcal{X}.$  Hence we have the bound
\begin{equation}
  \Pr_y[ \mathcal{C}_{x,y}]
  \leq
  \sum_{u \in \mathcal{U}} \exp(-f(\dM(u,x)-1)).
  \label{eq:degunion}
\end{equation}

As all the $1/2$ balls around $\mathcal{U}$ are disjoint, we have that for any $r>0$
\[
  \left|
  \mathcal{U} \cap \MB(m,r)
  \right|f(1/2)
  \leq \VolM(\MB(m,r)) = f(r).
\]
Therefore, we may estimate \eqref{eq:degunion} by partitioning the points into annuli of radius $r-1$ to $r,$ which yields
\begin{align*}
  \Pr_y[ \mathcal{C}_{x,y}]
  &\leq \sum_{r=1}^\infty \frac{f(r)}{f(1/2)} \exp( -f( \max(\dM(x,m),r)-2)). 
\end{align*}
It suffices to estimate the sum under the additional assumption that $\dM(x,m) >1$ by adjusting constants in \eqref{eq:connectiondecay}.
Subdivide the sum according to $r < \dM(x,m)$ and $r > \dM(x,m).$  For $r < \dM(x,m),$ we have
\[
  \sum_{r=1}^{\lfloor \dM(x,m) \rfloor} \frac{f(r)}{f(1/2)}
  e^{ -f( \max(\dM(x,m),r)-2)}
\leq
\dM(x,m) \frac{f(\dM(x,m))}{f(1/2)} e^{- f( \dM(x,m) -2)}.
\]
As for $r > \dM(x,m),$ recalling that $f(r)$ grows at least polynomially large in $r,$ the sum is no more than some absolute constant times its first term.  Combining these cases, we get that
\[
  \Pr_y[ \mathcal{C}_{x,y}]
  \leq C \dM(x,m) f(\dM(x,m)) \exp (-f (\dM(x,m) - 2))
\]
for some absolute constant $C>0.$  Applying \eqref{eq:volrecurrence} to the exponent, the desired \eqref{eq:connectiondecay} therefore follows for some other sufficiently large constant $C>0.$

We now turn to the second claim, that $(\mathscr{G},X_0)$ is reversible under the law $Q$.
  Let $X_0 = \rho$ and let $X_1$ be simple random walk on $\mathscr{G}$ after one step.  We will show that
  \[
    (\mathscr{G},X_0,X_1) 
    \lawequals
    (\mathscr{G},X_1,X_0), 
  \]
  as distributions on birooted equivalence classes of random graphs.  For this purpose, it suffices to show that for any $r \geq 1,$ and any finite rooted graphs $(g,v)$ and $(h,w)$
  \begin{equation*}
    Q\left[ 
      B_{\mathscr{G}}(X_0,r) 
    \cong g 
  \text{ and }
      B_{\mathscr{G}}(X_1,r) 
    \cong h
\right]
    =
    Q\left[ 
      B_{\mathscr{G}}(X_0,r) 
    \cong h 
  \text{ and }
      B_{\mathscr{G}}(X_1,r) 
    \cong g
\right],
  \end{equation*}
  where $\cong$ denotes equality up to rooted isomorphism.  Equivalently, it suffices to show that
  \begin{equation}
    \label{eq:fixedg}
    \frac{P\left[ 
      B_{\mathscr{G}}(X_0,r) 
    \cong g 
  \text{ and }
      B_{\mathscr{G}}(X_1,r) 
    \cong h
  \right]}
  {
    P\left[ 
      B_{\mathscr{G}}(X_0,r) 
    \cong h 
  \text{ and }
      B_{\mathscr{G}}(X_1,r) 
    \cong g
  \right]}
  =\frac{\deg v}{\deg w}.
  \end{equation}
  
  For clarity, let $\pi : \vso(\mathscr{G}) \to \mathcal{X}$ be the embedding of the vertices of $\mathscr{G}$ into $\mathbb{M}.$  Note that $\pi$ can be taken to be a function of the rooted isomorphism class $(\mathscr{G},\rho).$  Let $\tau$ be shorthand for $\tau_{x,\pi(X_1)}$, let $\mathcal{X}' = \tau(\mathcal{X}),$ and let $\mathscr{G}'$ be the dual graph of the Voronoi tessellation with nuclei $\mathcal{X}'.$  Let $\rho'$ be the vertex of $\mathscr{G}'$ whose embedding is at $x.$  Then $(\mathscr{G}',\rho')$ is isomorphic to $(\mathscr{G},X_1)$ as rooted graphs.  

  Let $\mathcal{E} \in \sigma(\mathcal{X},X_1)$ be any event on which $\deg \rho$ and $\deg X_1$ are both almost surely constants.  
  Denote these by $d_1$ and $d_2$ respectively.  Note that on $\tau^*(\mathcal{E}),$ $\deg \rho = d_2$ and $\deg X_1 = d_1.$
  By \eqref{eq:fixedg}, it suffices to show that 
  \[
    d_2 \Pr\left[ \tau^*(\mathcal{E}) \right] = d_1 \Pr\left[ \mathcal{E} \right].
  \]
  Let $\mu$ be the marginal probability measure of $\pi(X_1)$ on $\mathbb{M}.$  As the law $\mathcal{X}$ can be viewed as a tight Borel measure on a complete separable metric space
  (the boundedly finite measures under vague convergence, see \cite[Appendix A2.6]{VereJones}), we have the existence of a regular conditional probability measure $\Pr\left[\, \cdot\, \middle\vert \pi(X_1)=y \right].$  In particular, we may write
  \[
\Pr\left[ \mathcal{E} \right]
=
\int_{\mathbb{M}} \Pr\left[ \mathcal{E} \middle\vert \pi(X_1)=y \right]\,d\mu(y).
  \]

  Let $\Pr_{y}$ be the law of $\mathcal{X}$ conditioned to have a points at both $x$ and $y.$  Note that for any $y$ we have that $\Pr_y[ \pi(X_1) = y] > 0.$  
  Further, by a standard limiting argument, it is easily verified that 
  \[
\Pr\left[ \mathcal{E} \middle\vert \pi(X_1)=y \right]
=\frac{\Pr_y\left[ \mathcal{E} \cap \left\{ \pi(X_1) = y \right\} \right] }
{ 
  \Pr_y\left[ \pi(X_1) = y \right]
}
=
\Pr_y\left[ \mathcal{E} \middle\vert \pi(X_1)=y \right].
  \]
  Using that the degrees of $X_0$ and $X_1$ are specified on $\mathcal{E},$
  \begin{align*}
    \Pr_y\left[ \tau^*(\mathcal{E}) \cap \left\{\pi(X_1)=y\right\} \right]
    &=
    \Pr_y\left[ \tau_{x,y}^*(\mathcal{E}) \cap \left\{\pi(X_1)=y\right\} \right] \\
    &=
    \Pr_y\left[ \tau_{x,y}^*(\mathcal{E}) \cap \mathcal{C}_{x,y}\right]\frac{1}{d_2}, \\
    \intertext{where we have used that $X_1$ is simple random walk.  As $\tau_{x,y}$ is an isometry that interchanges $x$ and $y,$ we have $\Pr_y \circ \tau_{x,y}^* = \Pr_y.$  As $\mathcal{C}_{x,y}$ is invariant under $\tau^*_{x,y}$, we have} 
    \Pr_y\left[ \tau^*(\mathcal{E}) \cap \left\{\pi(X_1)=y\right\} \right]
    &=
    \Pr_y\left[ \mathcal{E} \cap \mathcal{C}_{x,y}\right]\frac{1}{d_2}, \\
    &=
    \Pr_y\left[ \mathcal{E} \cap \left\{\pi(X_1)=y\right\}\right]\frac{d_1}{d_2}.
  \end{align*}
Hence, integrating out the conditioning, we have
  \begin{align*}
d_2 \Pr\left[ \tau^*(\mathcal{E}) \right]
&=\int_{\mathbb{M}} d_2 \Pr\left[ \tau^*(\mathcal{E}) \middle\vert \pi(X_1)=y \right]\,d\mu(y)\\
&=\int_{\mathbb{M}} d_1 \Pr\left[ \mathcal{E} \middle\vert \pi(X_1)=y \right]\,d\mu(y)\\
&=d_1 \Pr\left[ \mathcal{E}\right].
\end{align*}
  
  \end{proof}

\begin{proposition}
  For either $G=\HPV$ or $\HPD,$ $(G,0)$ is a random weak limit of finite graphs.
  \label{prop:sofic}
\end{proposition}
\begin{remark}
  This was observed earlier by the first author and Oded Schramm, see~\cite[``Hyperbolic Surfaces'' proof of Theorem 6.2]{BenjaminiSchramm01}, before the notion of local limit was codified.  We elaborate on their idea here.
\end{remark}
\begin{proof}
  We will give the proof for $\HPV;$ the proof for $\HPD$ is identical. 
  Let $\{S_r\}_{r=1}^\infty$ be a family of compact Riemann surfaces so that $S_r$ has the property that any disk of radius $2r$ in $S_r$ is isometric to $\HB(0,2r) \subseteq \Htwo.$  Such a family of surfaces is known to exist, see~\cite[Proposition 1, Lemma 2]{Schmutz}.  Hence, on $S_r$ we can define a Poisson point process $\PPP_{S_r}$ whose intensity measure on any disk of radius $r$ is the pullback of the intensity of $\PPP$ on $\HB(0,2r).$  We can also associate to $\PPP_{S_r}$ its associated Voronoi tessellation, and we define $(G_r,\rho_r)$ to be the dual graph, where $\rho_r$ is a uniformly chosen vertex of $G_r.$  We claim that $(G,0)$ is the local limit of $(G_r,\rho_r),$ i.e. $(G,0)$ is the random weak limit of $G_r.$  To this end, let $G'_r$ be the induced subgraph of $G_r$ with vertices $\PPP_{S_r} \cap B_{S_r}(\rho_r, r).$ 

  As in the proof of Lemma~\ref{lem:triangletail}, the event that there is a vertex in $\PPP \cap \HB(0,r)$ which is in a triangle of diameter larger than $r$ is contained in the the event that there is empty disk of similar large radius.  In particular this event can be estimated by an event
  which is measurable with respect to $\PPP \cap \HB(0,2r).$  
Thus, the same argument shows that there is an event $\mathcal{E}_{S_r}$ for which
\[
  \left\{ \exists~v,w \in \PPP_{S_r} : d_{S_r}(\rho_r,v) \leq r, d_{S_r}(\rho_r,w) > 2r, \text{ and } d_{G_r}(v,w)=1 \right\} \subseteq \mathcal{E}_{S_r}
\]
that is measurable with respect to $\PPP_{S_r} \cap B_{S_r}(0,2r)$ and $\Pr(\mathcal{E}_{S_r}) = e^{-\omega(r)}.$ 

  
For any $r > 0,$ let $\PPP_r = \PPP \cap \HB(0,r),$ let $\HPV_r$ be the the induced subgraph of $\HPV$ on vertices $\PPP_r.$  Like with $\mathcal{E}_{S_r},$ we may find an event $\mathcal{E}_r$ so that
\[
  \left\{ \exists~v,w \in \PPP : \dH(0,v) \leq r, \dH(0,w) > 2r, \text{ and } d_{\HPV}(v,w)=1 \right\} \subseteq \mathcal{E}_{r}
\]
and $\Pr(\mathcal{E}_{r}) = e^{-\omega(r)}.$ 

For any finited rooted graph $(H,o)$ we thus have that
\[
  \Pr\left[ \left\{ (G_r',\rho_r) \cong (H,o) \right\} \cap \mathcal{E}_{S_r}^c \right]
  =
  \Pr\left[ \left\{ (\HPV_r,0) \cong (H,o) \right\} \cap \mathcal{E}_{r}^c \right].
\]
Observe that taking $r\to \infty,$ we have that      
\[
  \lim_{r \to \infty} B_{\HPV_r}(0,k) = B_{\HPV}(0,k) 
\]
almost surely, with the limit in the discrete topology.  Similarily, it is easy to see that $(G_r',\rho_r)$ contains $B_{G_r}(\rho_r, k)$ with probability going to $1$ as $r \to \infty$ for any fixed $k.$  Hence, we get that for any $k >0$ and any finite rooted graph $(H,o),$
\[
  \lim_{r\to\infty}
  \Pr \left[ B_{G_r}(\rho_r,k) \cong (H,o) \right] = 
  \Pr \left[ B_{\HPV}(0,k) \cong (H,o) \right],
\]
which completes the proof.
%


\end{proof}

\begin{proposition}
  For $G=\HPV,$  
  \[
    \limsup_{r\to \infty}|B_G(0,r)|^{1/r} < \infty,
  \]
  almost surely.
  \label{prop:expgrowth}
\end{proposition}

Proposition~\ref{prop:expgrowth} follows by a distance comparison.  Namely, the graph distance of  any nucleus to the origin is at least a fixed multiple of the hyperbolic distance, save for a finite number of exceptions.  To show this, we develop the following bound.
\begin{lemma}
  There are constants $\delta,\epsilon, r_0 > 0$ so that for all $r > r_0,$
  \[
    \Pr\left[
      \exists~x \in \PPP~:~r \leq \dH(0,x) < r+1 \text{ and } d_G(0,x) < \delta r
    \right] \leq e^{-\epsilon r}.
  \]
  \label{lem:dcompare}
\end{lemma}

\begin{proof}[Proof of Lemma~\ref{lem:dcompare}]
  Let $q > 1$ be a parameter to be determined later.  For each integer $j$ with $1 \leq j \leq \ell \Def \lceil (r+1)/q \rceil,$ define the annulus 
  \[
    A_j = \HB(0, j q) \setminus \overline{ \HB(0, (j-1)q) }.
  \]
  Subdivide each annulus $A_j$ into $\lceil \sinh(jq) \rceil$ equally sized Euclidean polar rectangles.  Let $\mathcal{D}$ be this collection of rectangles, and let $S_\ell$ be the set of all $\ell$-tuples $(d_i)_{i=1}^\ell$ of rectangles in $\mathcal{D}$ with the property that $d_i \subset A_i.$

  Let $\mathcal{E}$ be the event
  \begin{equation}
    \mathcal{E} = \left\{
      \exists\,(d_j)_1^\ell \in S_\ell~:~ | j : d_j \cap \PPP = \emptyset | \geq \ell/2 
    \right\}.
    \label{eq:gridevent}
  \end{equation}
  We claim that we can find some $q$ and $r_0$ large and some $\epsilon >0$ small so that with $r > r_0$
  \begin{equation}
    \Pr\left[
      \mathcal{E}
    \right] \leq e^{-\epsilon r}.
    \label{eq:gridbound}
  \end{equation}
  Set $p = \exp(-\lambda \min_{d \in \mathcal{D}} \Vol(d)).$  It can be verified that $\min_{d \in \mathcal{D}} \Vol(d) = \Omega(q).$ 
  By Chernoff bounds, for any fixed sequence $(d_j)_1^\ell \in S_\ell,$ 
  \begin{equation}
    \Pr\left[
      | j : d_j \cap \PPP = \emptyset | \geq \ell p + \ell/3 
    \right] \leq \exp\left( - \frac{\ell}{18 p} \right).
    \nonumber
  \end{equation}
  Hence, increasing $q$ so that $p < \tfrac 16$ and so that $1/(18pq) > 1+\epsilon,$ we have that
  \begin{equation}
    \Pr\left[
      | j : d_j \cap \PPP = \emptyset | \geq \ell/2 
    \right] 
    \leq \exp\left( - \frac{\ell}{18 p} \right)
    \leq \exp\left( -(1+\epsilon) r \right),
    \label{eq:chernoff}
  \end{equation}
  provided that $r$ is taken sufficiently large.
  It is easily verified that $|S_\ell| = O_q(e^r),$ and so \eqref{eq:gridbound} follows from a union bound.  

  Now suppose that $x \in \PPP$ is any point with $ r \leq \dH(0,x) < r+1.$  Let $0 = y_0, y_1,y_2, \ldots , y_m = x$ be the nuclei that appear in order in a path in $G$ from $0$ to $x$ that achieves the graph distance.  Recall that for any pair of Voronoi nuclei $y,z \in \PPP$ that are adjacent, there is a hyperbolic ball $B$ centered on the bisector of $y$ and $z$ and containing the geodesic between $y$ and $z$ so that $B \cap \PPP = \emptyset.$  All such hyperbolic balls contain one of two half-disks centered at $m$ of diameter $\dH(y,z)$ on either side of the geodesic from $y$ to $z,$ and hence one of these is empty as well.
  We will use this property to show that if $m$ is too small, there is a sequence $(d_i)_{i=1}^\ell \in S_\ell$ with too many empty rectangles. 

  Let $t_1, t_2, \ldots, t_{m'}$ be defined inductively by $t_1 = 1$ and 
  \[
    t_j = \min \{ i > t_{j-1}~:~ \dH(0,y_{i}) > \dH(0,y_{t_{j-1}}) \}.
  \]
  Let $r_i = \dH(0,y_{t_i}),$ so that $r_i$ form a strongly increasing sequence.  
  We will construct a disjoint family of disks $B_1, \ldots, B_{m'}$ with $m' \leq m$ so that 
  \begin{enumerate}
    \item $\sum_{i=1}^{m'} \diamH( B_i ) = \frac{1}{2}\max_i \dH(0,y_i),$
    \item $B_i$ is contained in the annulus centered at $0$ of outer radius $r_i$ and inner radius $r_{i-1}.$ 
    \item For all $1 \leq i \leq m',$ $B_i \cap \PPP = \emptyset.$
  \end{enumerate}
  For each $i \geq 1,$ we define $B_i$ as follows.  Consider the hyperbolic geodesic $g$ connecting $y_{t_i}$ and $y_{t_i - 1}.$  From the construction of $t_i,$ $y_{t_i - 1} \in \HB(0, r_{i-1}).$  Hence $g$ crosses $\partial \HB(0, (r_i + r_{i-1})/2)$ at some point $z.$  Let $V$ be the hyperbolic disk with diameter given by $g,$ and let $B$ be the disk $\HB(z, (r_i - r_{i-1})/2).$  Let $V'$ one of the half disks of $V$ on one of the sides of $g$ that contains no points of $\PPP.$  As $g$ contains a diameter of $B,$ $V' \cap B$ is also a half disk that no points of $\PPP.$  Within this half disk, there is a unique disk of half the diameter of $B,$ centered on the bisector of $y_{t_i-1}$ and $y_{t_i}.$  Let $B_i$ be this disk.  It is easily checked that all three properties hold for these choices of $B_i.$ 

  For any disk $B \subset \HB(0,r+1)$ we can look at the collection of recantangles $D(B) \subset \mathcal{D}$ that intersect the ray from $0$ to $z$ that are contained in $B.$  As these rectangles have radial hyperbolic diameter $q$ and each is contained in a $1$-neighborhood of some ray from $0$ to $\infty,$ we have there is some constant $c_q > 0$ so that $|D(B)| > \diamH(B)/q - c_q.$  
  
  Applying this construction to every ball $B_1, B_2, \ldots B_{m'},$ we have the existence of a sequence $(d_i)_{i=1}^\ell \in S_\ell$ so that at least
  \[
    \sum_{i=1}^{m'} (\diamH(B_i) / q - c_q) \geq r / q - c_qm
  \]
  of these balls are empty.  Hence on the complement of the event $\mathcal{E}$, we must have that $c_qm \geq r/2q.$  Taking $\delta < 1/2c_qq,$ the lemma therefore follows from~\eqref{eq:gridbound}.
\end{proof}

\begin{proof}[Proof of Proposition~\ref{prop:expgrowth}]

  Applying Borel-Cantelli together with the bound in {Lemma~\ref{lem:dcompare}}, we have that there is an $r_0 < \infty$ random and a $\delta >0$ so that for all $x \in \PPP$ with $\dH(0,x) > r_0,$
  \[
  d_G(0,x) \geq \delta \dH(0,x).
\]
The number of Voronoi cells intersecting $\HB(0,r_0)$ is almost surely finite, and thus there is an $r_1 > 0$ so that if $x \in \PPP$ satisfies $d_G(0,x) > r_1,$ then $\dH(0,x) > r_0.$  
Thus, for all $r > r_1$ the ball $B_G(0,r)$ is embedded in $\HB(0, r/\delta).$  

The number of Poisson points in $\HB(0, r/\delta)$ satisfies
\[
  \lim_{r \to \infty} \frac{|\PPP \cap \HB(0,r/\delta)|}{\VolH( \HB(0,r/\delta) ) } = \lambda
\]
almost surely.  Therefore,
\[
  \limsup_{r \to \infty} \frac{|B_G(0,r)|}{\VolH( \HB(0,r/\delta) ) } \leq \lambda,
\]
and the claim follows.
\end{proof}

\begin{proposition}
  With $G = \HPV$ or $\HPD,$ for almost evey realization of $G$, simple random walk started at $0$, considered as a process in the Poincar\'e disk, converges almost surely in the topology of $\C$ to a point on $S^1.$
  \label{prop:S1convergence}
\end{proposition}
\begin{proof}
  The proof here is a small modification of \cite[Theorem 4.1]{BenjaminiSchramm01}.
  For any point $z \in \C,$ we let $\theta(z) = \frac{z}{|z|}$ be the corresponding point in $S^1.$
  From the hyperbolic law of cosines, we can see that
  there are absolute constants $c_1,c_2 > 0$ so that
  for any two points $p,q \in \Htwo$ with $\dH(p,q) \leq \dH(q,0)$ with $\dH(q,0) \geq c_1,$
  \[
    \left| \theta(p) - \theta(q)\right|
    \leq c_2 e^{-\dH(0,q) + \dH(p,q)/2}.
  \]

  By Lemma~\ref{lem:triangletail} and Borel-Cantelli, we can show that for $G = \HPV$ or $G= \HPD,$ 
    \begin{equation}
      \label{eq:distortion}
    M = \sup_{x \in \PPP} \max_{ \substack{y \in \PPP,  \\ d_{G}(x,y) = 1}} \frac{ \dH(x,y)}{\log(2+\dH(0,x))} < \infty
  \end{equation}
  almost surely. 
  From the almost sure positive speed of $X_k,$ we have that there is a $c>0$ so that $d_G(X_k,0) \geq ck$ for all $k$ sufficiently large.  
  It follows that $\dH(X_k,0) \geq c' k/\log(k)$ for some other $c' > 0$ and all $k$ sufficiently large.  Hence, we get the estimate that
    \begin{align*}
    \left| \theta(X_{k+1}) - \theta(X_k)\right|
    &\leq \sup_{r > c' k/\log(k)} c_2 e^{-r + M\log(2+r)/2} \\
    &\leq c_2 e^{-c' k / \log(k) + M\log(2+k)/2}
  \end{align*}
  for all $k$ sufficiently large.  This is summable in $k$, and hence $\theta(X_k)$ converges almost surely.  As $\dH(X_k,0) \to \infty$ as well, the proof is complete.
\end{proof}

\section*{Acknowledgements.}
We would like to thank Pablo Lessa and Matias Piaggio for sharing an advance version of his work~\cite{Lasso} and for pointing out an error in an earlier draft.

\bibliographystyle{alphaabbr}
\bibliography{anchored}

\end{document}